\documentclass[EPiCempty]{easychair}

\usepackage[utf8]{inputenc}
\usepackage{amsmath,amssymb,mathtools}
\usepackage[table]{xcolor} % row/column colors in tables (\rowcolor, \cellcolor)
\usepackage{graphicx}
\usepackage{booktabs}
\usepackage{tabularx} % flexible-width tables to avoid overflow
\usepackage{caption} % \captionof for non-floating figures/tables
\usepackage{placeins} % provides \FloatBarrier
\usepackage{float}   % provides [H] “place exactly here” float specifier
\usepackage{cite}
\usepackage{tikz}
\usetikzlibrary{arrows.meta,calc,positioning}
\usepackage{colortbl}

\theoremstyle{plain}
\newtheorem{theorem}{Theorem}
\newtheorem{proposition}[theorem]{Proposition}
\newtheorem{lemma}[theorem]{Lemma}

\theoremstyle{remark}
\newtheorem{remark}[theorem]{Remark}
\theoremstyle{plain}

\newcommand{\R}{\mathbb{R}}
\newcommand{\N}{\mathbb{N}}
\newcommand{\E}{\mathbb{E}}
\newcommand{\inner}[2]{\langle #1, #2\rangle}
\newcommand{\Coop}{C}

\definecolor{qsblue}{HTML}{0072B2}   % OI blue
\definecolor{qsorange}{HTML}{E69F00} % OI orange
\definecolor{qsred}{HTML}{D55E00}    % OI vermillion
\definecolor{qsgreen}{HTML}{009E73}  % OI bluish green
\definecolor{qspurple}{HTML}{CC79A7} % OI reddish purple
\definecolor{qsgray}{HTML}{666666}   % neutral, for rules and inert elements

\tikzset{
  qsbox/.style={draw=qsgray, rounded corners=2pt, fill=black!2,
    align=center, inner xsep=7pt, inner ysep=5pt},
  qsarrow/.style={-{Stealth[length=2.2mm]}, line width=0.85pt},
  qsfast/.style={qsarrow, draw=qsorange},
  qsslow/.style={qsarrow, draw=qsblue},
  qsfeedback/.style={-{Stealth[length=2.0mm]}, draw=qsgreen,
    densely dashed, line width=0.85pt}
}

\title{Computing Extinction Barriers in a
Quorum-Sensing Reaction Network}

\author{
Mario Ayala
\and
Johannes Zimmer
}

\institute{
  School of Computation, Information and Technology,
  Technische Universit\"at M\"unchen,
  Boltzmannstra{\ss}e~3, 85748 Garching, Germany\\
  \email{mario.ayala@tum.de, jz@tum.de}
}

\authorrunning{Ayala and Zimmer}
\titlerunning{Computing Extinction Barriers}

\begin{document}
\raggedbottom % avoid large vertical stretch around floats (prevents big white gaps)
\maketitle

\begin{abstract}
We introduce a simple reaction-network model of a quorum-sensing population that couples the cell density $x$ to the signal density $w$. In the four-channel cell--signal network studied here, scaling signal production and removal by the same factor $r$ leaves all deterministic equilibria, and their stability types, unchanged. Nevertheless, we show that $r$ shifts the quasipotential barrier $\Delta V(r)$ for rare transitions towards extinction, and thus, under metastable exit assumptions, the mean time to reach a fixed neighbourhood of the extinction state on the exponential scale $e^{N\Delta V(r)}$.

We compute the barrier by minimization of the path action with the signal retained as a fluctuating coordinate, and we check it against mean extinction times obtained without the action solver, i.e., from the backward equation of the jump process on a truncated lattice and from exact stochastic simulation. Over a range of \(r\), the minimum-action barriers satisfy
\[
\Delta V(r)=\Delta V_\infty+O(1/r),
\]
where \(\Delta V_\infty\) is obtained by eliminating the signal first, and at \(r=1\) the barrier is \(55\%\) larger than \(\Delta V_\infty\). A formal first-order expansion attributes this increase to the lag of the signal, of which the fluctuations of the signal remove a quarter. The saddle barrier \(\Delta V(r)\) is also the least action needed to enter the basin of extinction, whereas the density threshold can be crossed more cheaply, by paths which keep the signal high and are usually followed by recovery. The mean extinction times of the backward equation, which include failed attempts, have a prefactor which tends to a constant, and they grow at a rate within \(10^{-5}\) of \(\Delta V(r)\) at every tested rate.
\end{abstract}

%=====================================================================
\section{Introduction}
\label{sec:intro}

Quorum sensing (QS) is a form of chemical communication in which cells release diffusible signalling molecules and respond to their concentration by altering gene expression. Among many others, QS regulates behaviours such as biofilm formation, exoenzyme production and bioluminescence, linking cellular activity to the surrounding population and its chemical environment \cite{miller2001quorum,waters2005quorum}. When these activities are costly to perform but benefit neighbouring cells, their regulation also shapes microbial cooperation \cite{diggle2007cooperation}.

The dynamics of this communication depend on both signal production and removal. Their speed introduces a timescale distinct from population growth and decline, raising the question of whether it can alter population persistence even when it leaves the deterministic equilibria in place. We address this question using a minimal, well-mixed stochastic model that explicitly retains both producing cells and the signal particles they release. More precisely, we investigate how this speed affects rare population extinction when the deterministic equilibrium locations and their stability types remain unchanged.

\subsection*{The model and the question}

In this work, instead of starting from a population-level ODE description, we formulate our model from basic biological events, as is standard in stochastic population biology and reaction-network theory~\cite{fournier2004microscopic,andersonkurtz2015}. Every cell is a producer, and cell division, cell death,
signal production and signal removal define a continuous-time Markov jump
process for the cell and signal counts $n^\Coop(t)$ and $n^w(t)$. At
population scale $N$, the density process is
\[
 U^N(t)=\bigl(X^N(t),W^N(t)\bigr)
 :=\left(\frac{n^\Coop(t)}{N},\frac{n^w(t)}{N}\right).
\]
Throughout, we write elements of $\R^2$ as tuples, i.e., $U=(x,w)$, and
identify them with column vectors, hence
$\dot U=F(U)$ below is an equation between column vectors and, for
$v\in\R^2$, $vv^{\!\top}$ is a $2\times2$ matrix.
At a given density value $(x,w)$, the increments and total rates of the four events are as follows.
\begin{center}
\small
\renewcommand{\arraystretch}{1.1}
\begin{tabular}{@{}lll@{}}
\toprule
\textbf{event} & \textbf{density increment} & \textbf{total rate}\\
\midrule
cell division     & $(1/N,0)$  & $Nx\,b(w)$\\
cell death        & $(-1/N,0)$ & $Nx\,d(x)$\\
signal production & $(0,1/N)$  & $Nr\alpha_\Coop x$\\
signal removal    & $(0,-1/N)$ & $Nr\kappa w$\\
\bottomrule
\end{tabular}
\end{center}
Here $b(w)$ is the division rate, which increases with the signal through a
Hill response, and $d(x)$ is the death rate, which includes a production cost
and crowding. The constant $\alpha_\Coop$ is the signal production rate per
cell, $\kappa$ is the signal removal coefficient, and $r>0$ multiplies both
signal channels. Hence $r$ sets the speed of the signal dynamics while
preserving the balance between production and removal at a fixed cell
density. Raising $r$ raises production while the cost $c$ is held fixed, so
varying $r$ is a controlled mathematical experiment on the timescale of the
signal, not an experimentally independent change of signal turnover; in a
model where the cost scales with production, $c$ and $r$ would be coupled. Section~\ref{sec:network} gives the closed forms of $b$ and $d$, the
generator, the biological rationale for the rates, and the dimensionless
parameter values used below. These values are illustrative; our study is
methodological rather than organism-specific.

\subsection*{One process, three regimes}

At fixed $r$, the event rules above supply three complementary
large-population descriptions: the law of large numbers (LLN), the central
limit theorem (CLT) and the large-deviation principle (LDP), given by
\begin{equation}
 \label{eq:intro-regimes}
 \begin{aligned}
 \text{LLN:}\quad
   &U^N\longrightarrow U,
   &&\dot U=F(U),\\
 \text{CLT:}\quad
   &\sqrt N\,(U^N-U)\Longrightarrow Z,
   &&dZ=\nabla F(U)\,Z\,dt+D(U)^{1/2}dW,\\
 \text{LDP:}\quad
   &\mathbb P(U^N\approx\phi)
     \asymp \exp\{-N\mathcal S_T(\phi)\},
   &&\mathcal S_T(\phi)=\int_0^T\mathbb L(\phi,\dot\phi)\,dt.
 \end{aligned}
\end{equation}
Here $F$ is the mean-field drift (written out below), $\nabla F$ its linearization along the
deterministic solution, $D$ the diffusion matrix the reaction channels supply,
$W$ a standard planar Brownian motion, and $\mathbb L$ the Legendre transform
of the limiting Hamiltonian. The LLN gives the deterministic trajectory, the
CLT resolves fluctuations of order $N^{-1/2}$ around it, and the LDP assigns
to a \emph{rare} path $\phi$ a probability on the
exponential scale $\exp\{-N\mathcal S_T(\phi)\}$
\cite{feng2006large,elgart2004rare}.

The descriptions in \eqref{eq:intro-regimes} retain different information about the event rules.
The deterministic drift combines opposing microscopic mechanisms through their net rates,
whereas the fluctuation diffusion and path action also depend on the
individual rates. Hence neither the deterministic drift nor its equilibria,
described below, determine the cost of extinction.

In order to explain this in more detail, we start from the LLN. By Kurtz's theorem \cite{kurtz1970solutions} (Theorem~\ref{thm:lln} in this work), $U^N$ converges to a deterministic limit $U=(x,w)$ satisfying the following ODE system
\[
 \dot U = F(U)
 =
 \begin{pmatrix}
 x\bigl[b(w)-d(x)\bigr]\\
 r\bigl[\alpha_\Coop x-\kappa w\bigr]
 \end{pmatrix}.
\]
Therefore, at equilibrium, the signal satisfies
\[
 w=\bar w(x):=\frac{\alpha_\Coop x}{\kappa},
\]
and every positive equilibrium solves $b(\bar w(x))=d(x)$, independently
of $r$. For the bistable parameters studied here, the system has three
equilibria: the extinction state $U_{\mathrm{off}}=(0,0)$, a positive saddle
state $U_\ast=(x^\ast,\bar w(x^\ast))$, and a stable cooperative state
$U_{\mathrm{on}}=(x_{\mathrm{on}},\bar w(x_{\mathrm{on}}))$.
Proposition~\ref{prop:r-independence} shows that all three
equilibria retain their stability types as $r$ varies. Nevertheless, their
relaxation rates, trajectories and basins need not remain unchanged.

\subsection*{Extinction}
In this work, by population extinction we mean a transition from the stable
cooperative state $U_{\mathrm{on}}$ towards $U_{\mathrm{off}}$. The deterministic attraction basins of these states are
separated by a curve, the separatrix, whose location in general depends on $r$,
even though the equilibria themselves do not. The transition therefore involves the
joint cell-signal state and cannot be identified by a cell-density threshold
alone.

In order to measure this transition, we choose a fixed small neighbourhood
$R_{\mathrm{off}}$ (see \eqref{eq:def-R-off})
 of the extinction state $U_{\mathrm{off}}$, contained in its
deterministic attraction basin, and define, for the process started near
$U_{\mathrm{on}}$,
\begin{equation}
\label{eq:collapse-time-off}
\tau_N^{\mathrm{off}}(r)
:=\inf\{t\geq0:U^N(t)\in R_{\mathrm{off}}\},
\end{equation}
i.e., the first arrival in a region from which the deterministic evolution tends
to $U_{\mathrm{off}}$. We call $\tau_N^{\mathrm{off}}(r)$ the extinction time. Entry into
$R_{\mathrm{off}}$ is not yet cell extinction, $X^N=0$, and does not exclude
subsequent stochastic recovery. We quantify persistence by means of the
mean of this stopping time and its dependence on the population scale $N$. The
target set is specified in more detail in Section~\ref{sec:off-target}. The justification for
using the saddle action as the transition barrier, and the assumptions needed
to connect that barrier to mean extinction times, are given in
Section~\ref{sec:barrier}. Section~\ref{sec:ssa} states which stopping times are computed.

%%%%%%%%%%%%%%%%%%%%%%%%%%%%%%%%
For large $r$, it is natural at the deterministic level to replace the
signal by $\bar w(x)$. Whether this reduction also preserves the probability
of rare extinction is a separate question. At finite $r$, a fluctuating path
towards extinction can explore the joint cell-signal plane: the signal need not
remain at its quasi-steady value while the cell density falls. The quantity
we compute measures the cost of this rare transition from the cooperative
state $U_{\mathrm{on}}$ to the saddle $U_\ast$, which is the least-cost
entrance to the basin of extinction (Section~\ref{sec:barrier}). Figure~\ref{fig:mechanism} sketches the basic
mechanism.

\begin{figure}[tbp]
\centering
\begin{tikzpicture}[x=0.73cm,y=0.68cm,font=\small]
  \draw[-{Stealth[length=2mm]},draw=qsgray] (0,0) -- (6.9,0);
  \node[below] at (5.75,0) {cell density $x$};
  \draw[-{Stealth[length=2mm]},draw=qsgray] (0,0) -- (0,4.8);
  \node[rotate=90] at (-0.60,2.35) {signal density $w$};
  \draw[densely dashed,draw=qsgray,line width=0.8pt]
    (0.45,0.45) -- (6.75,4.03);
  \node[below right=-1pt] at (3.45,2.15) {$w=\bar w(x)$};

  \coordinate (ustar) at (2.0,1.33);
  \coordinate (uon) at (6.0,3.60);
  \filldraw[fill=white,draw=black,line width=0.9pt] (ustar) circle (2.3pt);
  \filldraw[fill=black,draw=black] (uon) circle (2.3pt);
  \node[below right=-1pt] at (ustar) {$U_\ast$};
  \node[above right=-1pt] at (uon) {$U_{\mathrm{on}}$};

  \draw[-{Stealth[length=2.1mm]},draw=qsgray,line width=1.05pt]
    (uon) -- (ustar);
  \draw[-{Stealth[length=2.2mm]},draw=qsred,line width=1.25pt]
    (uon) .. controls (5.20,4.43) and (3.10,3.44) .. (ustar);

  \draw[-{Stealth[length=1.8mm]},draw=qsblue,line width=0.9pt]
    (3.05,1.93) -- (4.25,2.61);
  \node[qsbox,text width=5.1cm,align=left] at (10.85,2.55) {
    \tikz\draw[-{Stealth[length=1.8mm]},draw=qsred,line width=1.2pt]
      (0,0)--(0.72,0);\; finite-$r$ minimum-action path\\[3pt]
    \tikz\draw[-{Stealth[length=1.8mm]},draw=qsgray,line width=1.0pt]
      (0,0)--(0.72,0);\; eliminated-signal path\\[3pt]
    \tikz\draw[-{Stealth[length=1.8mm]},draw=qsblue,line width=0.9pt]
      (0,0)--(0.72,0);\; deterministic flow\\[5pt]
    % \centering signal lag: $w>\bar w(x)$\\[2pt]
    % \centering $\Delta V(r)>\Delta V_\infty$
  };
\end{tikzpicture}
\caption{Why signal elimination changes the barrier without changing the
deterministic equilibria. Both paths run from $U_{\mathrm{on}}$ to the saddle
$U_\ast$, the endpoint of the barrier (Section~\ref{sec:barrier}). Elimination
confines the path to the dashed relation $w=\bar w(x)$; at finite $r$ the
signal lags above it, so the minimum-action path joins the same endpoints
through different states.}
\label{fig:mechanism}
\end{figure}

We denote the least action from $U_{\mathrm{on}}$ to $U_\ast$ by
$\Delta V(r)$, the quasipotential barrier; its precise definition is given
in Section~\ref{sec:barrier}. Subject to metastable exit estimates for the
jump process, which we state but do not prove, this barrier governs the mean
of the extinction time \eqref{eq:collapse-time-off} on the exponential scale
\[
 \E[\tau_N^{\mathrm{off}}(r)]\asymp \exp\{N\Delta V(r)\}.
\]
We compare it with $\Delta V_\infty$, the barrier obtained by eliminating
the signal before constructing the reduced path action.

\FloatBarrier

\subsection*{Contribution and main findings}

The limit theorems and the large-deviation principle are established tools.
Our contribution is a calibrated barrier computation for this four-channel
network, checked independently: we minimize the path action with the signal
retained as a fluctuating coordinate, show that the saddle barrier is the least
interior action needed to enter the basin of extinction
(Proposition~\ref{prop:saddle-off}), and compare with mean extinction times
obtained from the backward equation of the jump process and from exact
stochastic simulation, neither of which uses the action solver. The
minimum-action results are compared with the expansion
\begin{equation}
 \label{eq:crossover-law}
 \Delta V(r)=\Delta V_\infty+\frac{A}{r}+O(r^{-2}),
 \qquad r\to\infty,
\end{equation}
whose coefficient $A$ is the prediction of a formal boundary-layer calculation.
We sketch this calculation in Section~\ref{ssec:finite-r} and refer
to~\cite{ayala2026cooperation} for the details; we do not prove the expansion
here. Over the values of $r$ studied, the minimum-action results support the
$1/r$ correction, and at $c=0.36$ and $r=1$ the computed barrier is $55\%$
larger than $\Delta V_\infty$. The same calculation separates the lag of the
signal, which raises the barrier, from the fluctuations of the signal, which
remove a quarter of this increase.

For the extinction time we solve the backward equation of the jump process on a
truncated lattice, for $40\le N\le400$. The prefactor
$\E[\tau_N^{\mathrm{off}}(r)]\,e^{-N\Delta V(r)}$ tends to a constant, with a
correction of order $1/N$, and the fitted growth rate differs from the
minimum-action barrier by less than $10^{-5}$ at each of the five rates. This
separates it from $\Delta V_\infty$ at these rates, and the simulated means agree
with the exact ones within sampling error. The density threshold $x^\ast$
behaves differently: at $r=0.5$ a threshold point with higher signal costs
$7.7\%$ less action than the saddle, and in simulation most threshold crossings
at this rate are followed by recovery. Since the barrier enters the mean
extinction time multiplied by $N$, an $r$-dependent barrier changes the mean
extinction time by a factor exponential in $N$, although the equilibria do not
move.

That the elimination of a fast variable can change the rate of a rare event is well known. Large deviations of fast-slow systems are analysed
in~\cite{bouchet2016large}, the elimination of fast stochastic variables at the
level of the linear noise approximation in~\cite{thomas2012rigorous}, and the
breakdown of fast-slow reductions for escape rates in~\cite{newby2013breakdown}.
In population dynamics, WKB methods for extinction are reviewed
in~\cite{assaf2017wkb}, and~\cite{khasin2009extinction} shows that an extinction
rate can change exponentially under a small change of the parameters. The
closest setting to ours is the genetic switch of~\cite{assaf2011determining},
where the fast species is the mRNA: there the mRNA noise changes the action at
leading order, through bursts of protein production, and the correction in the
ratio of the two lifetimes is stated as an order of magnitude.
In the network we study here, the fast variable is as abundant as the slow one, hence
the leading order is the eliminated-signal barrier $\Delta V_\infty$ and the
effect of the signal starts at order $1/r$. The minimum action method and the
backward equation are standard tools. Relative to these works, the present paper
gives a quantified example in a quorum-sensing network: the barriers at finite
$r$ with an independent check by the backward equation, the numerical test of
the coefficient $A$, of its separation into lag and noise and of the range of
$r$ in which the $1/r$ law holds, and the comparison between the threshold, the
recovery and the basin pictures of extinction. We are not aware of a formula for
the first-order coefficient in the cited works. The
formula~\eqref{eq:A-theory} is obtained in the forthcoming paper~\cite{ayala2026cooperation}, and its
derivation is not claimed as a contribution of this paper.

\subsection*{Outline}

Section~\ref{sec:network} gives the reaction-network formulation and
parameters. Section~\ref{ssec:limits} develops the deterministic dynamics,
Gaussian fluctuations and large deviations.
Section~\ref{sec:barrier} defines the extinction barrier, relates it to the
basin of $U_{\mathrm{off}}$, and obtains the eliminated-signal benchmark
$\Delta V_\infty$.
Section~\ref{sec:numerics} presents the action minimization, the backward
equation and the simulations, Section~\ref{sec:results} compares their
results, and Section~\ref{sec:conclusion} discusses the conclusions and
remaining limitations.

%=====================================================================
\section{The reaction network}
\label{sec:network}

\subsection{Events, state, and generator}
\label{ssec:biology}

We now describe in detail the reaction network introduced in Section~\ref{sec:intro}. Signal
increases the cell division rate through a sigmoidal response, production
contributes a private cost to cell death, and crowding limits growth.
Particles are created and destroyed one at a time, and no event changes both
populations at the same time. The factor $r$ multiplies production and removal
together, so it changes the speed of the signal dynamics without changing the
stationary signal level (Proposition~\ref{prop:r-independence}).

The state is the density $U=(x,w)$ of the introduction, with $x=n^\Coop/N$
and $w=n^w/N$, and the system is assumed well mixed throughout. Each of the
four distinct mechanisms changes exactly one coordinate of $U$ by $1/N$, while the total
event rate is of order $N$. The resulting density process
has a deterministic large-population limit and rare-event probabilities on the
exponential scale $N$. Notice that cells and signal share the population scale
$N$, i.e., the signal is as abundant as the cells. This is a modelling choice,
and Section~\ref{sec:results} states what changes when the signal count is
scaled as $\nu N$ instead.

We model these mechanisms as the reaction channels of a stochastic reaction
network~\cite{andersonkurtz2015}: a channel $\varrho$ is one reaction,
specified by a fixed increment $\Delta_\varrho$ and a rate density
$\beta_\varrho$. The time-evolution of the state $U$ is therefore a
continuous-time Markov jump process with state space
$(N^{-1}\N_0) \times (N^{-1}\N_0)$: each channel carries an exponential clock
of rate $N\beta_\varrho(U)$ depending on the current state $U$, and its firing
moves the state by a fixed increment
\[
 U\longmapsto U+\frac{\Delta_\varrho}{N} .
\]
The channels, their increments, and their rate densities are summarized in
Figure~\ref{fig:network}.

\begin{figure}[tbp]
\centering
\begin{tikzpicture}[x=1cm,y=1cm,font=\small]
  \node[circle,draw=qsblue,line width=1pt,fill=qsblue!8,
    minimum size=12mm,align=center] (cell) at (3.15,1.75)
    {$\Coop$\\[-2pt]{\scriptsize cells $x$}};
  \node[circle,draw=qsorange,line width=1pt,fill=qsorange!10,
    minimum size=12mm,align=center] (signal) at (8.15,1.75)
    {$w$\\[-2pt]{\scriptsize signal}};
  \node (sinkc) at (0.25,-0.10) {$\varnothing$};
  \node (sinkw) at (11.05,-0.10) {$\varnothing$};

  \draw[qsslow] (cell.145) .. controls +(0.0,1.40) and +(0.0,1.40) ..
    node[above,align=center] {division $\Coop\to2\Coop$\\
      $xb(w)$, $\Delta=(+1,0)$} (cell.35);
  \draw[qsslow] (cell.south west) --
    node[midway,above left=1pt,align=center] {death\\
      $xd(x)$, $\Delta=(-1,0)$} (sinkc);
  \draw[qsfast] (cell.east) --
    node[below=2pt,align=center] {production $\Coop\to\Coop+w$\\
      $r\alpha_\Coop x$, $\Delta=(0,+1)$} (signal.west);
  \draw[qsfast] (signal.south east) --
    node[midway,above right=1pt,align=center] {removal\\
      $r\kappa w$, $\Delta=(0,-1)$} (sinkw);
  \draw[qsfeedback] (signal.north west) to[bend right=20]
    node[above,align=center,text=qsgreen] {signal-enhanced\\division $b(w)$}
    (cell.north east);

  \node[draw=qsorange,fill=qsorange!6,rounded corners=2pt,
    inner xsep=7pt,inner ysep=3pt,align=center] at (6.8,-0.70)
    {$r$ multiplies both signal channels};
\end{tikzpicture}
\caption{The four stochastic channels and their feedback. Blue arrows change
the cell count and orange arrows change the signal count. The dashed arrow is
regulation, not an additional jump: signal enters the division rate $b(w)$.}
\label{fig:network}
\end{figure}

Notice that our model is microscopic: it specifies which events occur and at what
rate, and it imposes nothing at the population level. The dynamics of $U$ are best described in terms of its infinitesimal generator, which acts
on functions of the density state by
\begin{equation}
 \label{eq:generator}
 (\mathfrak L_Nf)(U)
 =\sum_{\varrho}N\beta_\varrho(U)
  \left[f\!\left(U+\frac{\Delta_\varrho}{N}\right)-f(U)\right],
\end{equation}
i.e., the time-evolution of $U$ is a density-dependent Markov chain in the sense of
Kurtz~\cite{kurtz1970solutions,ethier2009markov}. Everything computed below is obtained from \eqref{eq:generator}: the
large-population descriptions of Section~\ref{ssec:limits} are consequences of it
rather than starting points, and so is the barrier of
Section~\ref{sec:barrier}.

\subsection{Rates and parameters}
\label{ssec:params}
Each of the three functional forms below is chosen for a reason rather than
for tractability. The division response is of Hill type because quorum-sensing
regulation is cooperative: a receptor of LuxR type binds the signal and the
complex multimerizes before it binds the promoter, so the response to signal
is switch-like rather than graded, and the Hill exponent $h$ is what carries
that cooperativity~\cite{miller2001quorum,waters2005quorum,hense2007does}.
Production enters the death rate because making signal is metabolically
costly, which is also what makes the behaviour exploitable by
non-producers~\cite{diggle2007cooperation}. The remaining term in the death
rate is ordinary logistic self-limitation. Hence
\begin{equation}
 \label{eq:rates}
 b(q)=b_0+b_1\frac{q^h}{K_h^h+q^h},
 \qquad
 d(x)=d_0+c+\theta x .
\end{equation}
The sigmoidal response in \eqref{eq:rates} permits bistability for the
dimensionless values of Table~\ref{tab:parameters}, which are used throughout.
They are illustrative, hence the conclusions below are methodological rather
than organism-specific.

\begin{table}[tbp]
\caption{Parameters of the four-channel reaction network; the increments and rates of the channels are those of the table in the introduction. The cost $c$ is varied where stated, with $c=0.36$ as reference, and $r$ is swept over two decades. All quantities are dimensionless.}
\label{tab:parameters}
\centering
\small
\renewcommand{\arraystretch}{1.12}
\setlength{\tabcolsep}{3.5pt}

% tabularx prevents overflow by letting the "meaning" column wrap.
\begin{tabularx}{0.8\textwidth}{@{}l l >{\raggedright\arraybackslash}X r@{}}
\toprule
\textbf{channel $\varrho$} & \textbf{symbol} & \textbf{meaning} & \textbf{value}\\
\midrule
\rowcolor{qsblue!10}
\textcolor{qsblue}{\textbf{division}} & $b_0$ & basal division rate & $0.5$\\
\rowcolor{qsblue!10}
& $b_1$ & maximal signal-induced increment & $2.0$\\
\rowcolor{qsblue!10}
& $K_h$ & half-response signal level & $1.0$\\
\rowcolor{qsblue!10}
& $h$ & Hill exponent (cooperativity) & $4$\\
\addlinespace[2pt]
\rowcolor{qsblue!10}
\textcolor{qsblue}{\textbf{death}} & $d_0$ & basal death rate & $0.8$\\
\rowcolor{qsblue!10}
& $c$ & cost of production & $0.36$\\
\rowcolor{qsblue!10}
& $\theta$ & crowding coefficient & $1.0$\\
\addlinespace[2pt]
\rowcolor{qsorange!12}
\textcolor{qsorange}{\textbf{production}} & $\alpha_\Coop$ & production rate per cell & $2.0$\\
\addlinespace[2pt]
\rowcolor{qsorange!12}
\textcolor{qsorange}{\textbf{removal}} & $\kappa$ & signal removal rate & $1.0$\\
\midrule
\rowcolor{qsgreen!10}
\textcolor{qsgreen}{\textbf{signal channels}} & $r$ & common factor of production and removal & varied\\
\bottomrule
\end{tabularx}
\end{table}

%=====================================================================
\section{Three regimes}
\label{ssec:limits}

The large $N$ behaviour of the process admits three descriptions, distinguished by the
size of the deviation from the deterministic path that each one resolves: a law
of large numbers for the path itself, a central limit theorem for deviations of
order $N^{-1/2}$, and a large-deviation principle for deviations of order one.
All three are known results in the context of reaction networks. We state them for the network of
Figure~\ref{fig:network} and verify their hypotheses, and we cite rather than
reproduce their proofs. The barrier computed in this paper is an object of the
third regime, and \eqref{eq:intro-regimes} of the introduction states all
three side by side.

\subsection{Law of large numbers}

Write $F(U)=\sum_\varrho\Delta_\varrho\beta_\varrho(U)$ for the drift associated
with \eqref{eq:generator}; summing the four increments of
Figure~\ref{fig:network} against their rate densities gives the vector field
of the introduction.

\begin{theorem}[Kurtz~\cite{kurtz1970solutions}; see also
{\cite[Ch.~11]{ethier2009markov}}]
\label{thm:lln}
Fix $r>0$, let $G\subset(0,\infty)^2$ be open and bounded, and suppose
$U^N(0)\to U_0\in G$ in probability. Let $U=(x,w)$ solve $\dot U=F(U)$, i.e.,
\begin{equation}
 \label{eq:mean-field}
 \dot x=x\bigl[b(w)-d(x)\bigr],
 \qquad
 \dot w=r\bigl[\alpha_\Coop x-\kappa w\bigr],
\end{equation}
with $U(0)=U_0$, and let $T>0$ be such that $U(t)\in G$ for every $t\le T$.
Then for every $\varepsilon>0$,
\[
 \lim_{N\to\infty}\mathbb P
 \Bigl(\sup_{t\le T}\lvert U^N(t)-U(t)\rvert>\varepsilon\Bigr)=0 .
\]
\end{theorem}

The application of~\cite{kurtz1970solutions} is immediate here: the family is
finite, the increments $\Delta_\varrho$ have unit length, and, since $h$ is a
positive integer and $K_h^h+w^h$ does not vanish, every rate density
$\beta_\varrho$ of Figure~\ref{fig:network} is smooth on $[0,\infty)^2$ and
therefore Lipschitz on compact subsets of it. Notice that $G$ is required to
be bounded because $xd(x)$ grows quadratically; the restriction is harmless
below, where every trajectory of interest stays in a compact subset of
$(0,\infty)^2$.

\subsubsection{Equilibria and stability}
Within this mean-field system, the formal limit $r\to\infty$ slaves the signal
to
\begin{equation}
 \label{eq:slaved}
 \bar w(x)=\frac{\alpha_\Coop x}{\kappa}.
\end{equation}
The order of this statement matters. At fixed $N$, the fast signal has a
conditional Poisson equilibrium rather than the deterministic value
\eqref{eq:slaved}; deterministic slaving is the fast limit after the
mean-field reduction.
In this subsection, however, $\bar w$ serves only to locate and classify the
equilibria of the full system \eqref{eq:mean-field}, not to reduce it; the
dynamics with the signal eliminated are taken up in
Section~\ref{ssec:eliminated}.
Writing $\bar b(x):=b(\bar w(x))$ for the division rate
along \eqref{eq:slaved} and $\bar F(x):=x[\bar b(x)-d(x)]$ for the slaved drift,
the factor $r$ leaves the deterministic picture untouched.

\begin{proposition}
\label{prop:r-independence}
Let $\theta>0$ and let $U_{\mathrm{eq}}=(x_{\mathrm{eq}},w_{\mathrm{eq}})$ be an
equilibrium of \eqref{eq:mean-field}. Then
$w_{\mathrm{eq}}=\bar w(x_{\mathrm{eq}})$ and $\bar F(x_{\mathrm{eq}})=0$, both
independently of $r$. Moreover, the Jacobian $J_r$ of
\eqref{eq:mean-field} at $U_{\mathrm{eq}}$ satisfies

\begin{enumerate}\renewcommand{\labelenumi}{(\roman{enumi})}
\item If $x_{\mathrm{eq}}>0$, then
\[
 \det J_r=-r\kappa\,\bar F'(x_{\mathrm{eq}}),
 \qquad
 \operatorname{tr}J_r=-x_{\mathrm{eq}}d'(x_{\mathrm{eq}})-r\kappa<0 ,
\]
hence for every $r>0$ the equilibrium is a saddle when
$\bar F'(x_{\mathrm{eq}})>0$ and asymptotically stable when
$\bar F'(x_{\mathrm{eq}})<0$.

\item If $x_{\mathrm{eq}}=0$, then $U_{\mathrm{eq}}=U_{\mathrm{off}}:=(0,0)$ and
$J_r$ has eigenvalues $b_0-d_0-c$ and $-r\kappa$, hence for every $r>0$ the
extinction state is asymptotically stable when $b_0<d_0+c$ and a saddle when
$b_0>d_0+c$.
\end{enumerate}
\end{proposition}

\begin{proof}
The second component of \eqref{eq:mean-field} vanishes if and only if
$w=\bar w(x)$, and the first then reads $\bar F(x)=0$; neither involves $r$.
(i) Using $\bar b(x_{\mathrm{eq}})=d(x_{\mathrm{eq}})$,
\[
 J_r=\begin{pmatrix}
 -x_{\mathrm{eq}}d'(x_{\mathrm{eq}}) & x_{\mathrm{eq}}b'(w_{\mathrm{eq}})\\
 r\alpha_\Coop & -r\kappa
 \end{pmatrix},
 \qquad
 \bar F'(x_{\mathrm{eq}})
 =x_{\mathrm{eq}}\bigl[(\alpha_\Coop/\kappa)b'(w_{\mathrm{eq}})-d'(x_{\mathrm{eq}})\bigr],
\]
whence $\det J_r=-r\kappa\bar F'(x_{\mathrm{eq}})$, and the trace is negative
since $d'\equiv\theta>0$. (ii) At the origin $J_r$ is lower triangular with
diagonal entries $b_0-d_0-c$ and $-r\kappa$.
\end{proof}

For Table~\ref{tab:parameters} at $c=0.36$ the system has the three equilibria of
Figure~\ref{fig:nullclines}: the stable extinction state $U_{\mathrm{off}}$, an unstable
threshold $x^\ast$, and a stable cooperative state $x_{\mathrm{on}}$, with
\[
 x^\ast=0.55881,
 \qquad
 x_{\mathrm{on}}=1.29675,
 \qquad
 \bar w(x_{\mathrm{on}})=2.59350 .
\]
Since $b_0<d_0$ in Table~\ref{tab:parameters}, $U_{\mathrm{off}}$ is
asymptotically stable for every cost $c\ge0$ and every $r>0$.
%%%%%%%% R off %%%%%%%%%
\subsubsection{Basins of attraction and the extinction target}
\label{sec:off-target}

Let $\Phi_t^{(r)}(u)$ denote the deterministic solution starting
from $u$, and write $Q=(0,\infty)^2$. The two attraction basins
in the open quadrant are
\[
 \mathcal B_{\mathrm{on}}(r)
 :=\{u\in Q:\Phi_t^{(r)}(u)\to U_{\mathrm{on}}\},
 \qquad
 \mathcal B_{\mathrm{off}}(r)
 :=\{u\in Q:\Phi_t^{(r)}(u)\to U_{\mathrm{off}}\}.
\]
For Table~\ref{tab:parameters} and each cost $c\in[0.2,0.5]$, which contains
the costs used below, $b_0<d_0+c<b_0+b_1$. The positive zeros of the slaved
drift $\bar F$ are the positive roots of
\[
 P(x):=\bigl(b_0-d_0-c-\theta x\bigr)\bigl(K_h^h+(\alpha_\Coop x/\kappa)^h\bigr)
 +b_1(\alpha_\Coop x/\kappa)^h,
\]
obtained by multiplying $b(\bar w(x))=d(x)$ by the positive factor
$K_h^h+(\alpha_\Coop x/\kappa)^h$. For $h=4$ the polynomial $P$ has degree $5$,
and its nonzero coefficients, in the degrees $0,1,4,5$, have the signs
$-,-,+,-$. By Descartes' rule of signs $P$ has at most two positive roots,
counted with multiplicity. Since $P(0)<0$, $P(1)=9.9-17c>0$ and
$P(x)\to-\infty$ as $x\to\infty$, it has exactly two, $x^\ast<x_{\mathrm{on}}$,
and both are simple. The equilibria in $[0,\infty)^2$ are then exactly
$U_{\mathrm{off}}$, the saddle $U_\ast=(x^\ast,w^\ast)$, where
$w^\ast:=\bar w(x^\ast)$, and $U_{\mathrm{on}}$. Since $\bar F<0$ on
$(0,x^\ast)$ and both zeros are simple, $\bar F'(x^\ast)>0>\bar
F'(x_{\mathrm{on}})$, hence by Proposition~\ref{prop:r-independence}, for
every $r>0$, $U_\ast$ is a hyperbolic saddle and $U_{\mathrm{on}}$,
$U_{\mathrm{off}}$ are asymptotically stable.

The two basins are separated by the stable manifold $W^s(U_\ast)$ of the
saddle, the separatrix, i.e., the set of initial states whose solutions
converge to $U_\ast$. Although the equilibria are independent of $r$, the
separatrix generally is not (Figure~\ref{fig:nullclines}). Thus the cell
density alone does not determine which equilibrium a trajectory approaches.
Two pieces of the picture are nevertheless fixed for all $r$. One can show
that, for every $r>0$, the quadrant
$\{(x,w):x>x^\ast,\ w>w^\ast\}$ lies in $\mathcal B_{\mathrm{on}}(r)$ and the
segment $\{x^\ast\}\times(0,w^\ast)$ lies in $\mathcal B_{\mathrm{off}}(r)$;
we omit the proof. In particular, the threshold line $\{x=x^\ast\}$ meets the
separatrix only at $U_\ast$, and a crossing of that line with $w>w^\ast$ lies
in the cooperative basin.

\begin{figure}[tbp]
\centering
\includegraphics[width=0.6\textwidth]{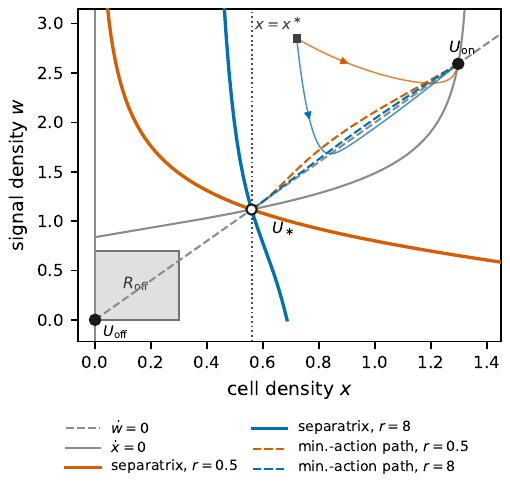}
\caption{Phase portrait of \eqref{eq:mean-field} at $c=0.36$. Grey: the nullclines $\dot w=0$ (dashed) and $\dot x=0$ (solid), which do not depend on $r$; filled points are the stable states and the open point is the saddle $U_\ast$. Thick curves: the separatrix at $r=0.5$ and $r=8$, with $\mathcal B_{\mathrm{off}}(r)$ below and to the left of it. Dotted: the threshold line $x=x^\ast$. Shaded: the target $R_{\mathrm{off}}=[0,0.3]\times[0,0.7]$. Dashed coloured curves: minimum-action paths from $U_{\mathrm{on}}$ to $U_\ast$. Thin curves: relaxation paths from the square marker, in the same colours, which reach the same state along routes that depend on $r$. All quantities are dimensionless.}
\label{fig:nullclines}
\end{figure}

To measure a transition towards extinction, we record arrival
in a specified neighbourhood of $U_{\mathrm{off}}$ within its
attraction basin. This places the stopping target beyond the
separatrix, rather than at the first crossing of it.
Specifically, choose
\begin{equation}
 \label{eq:def-R-off}
 R_{\mathrm{off}}:=[0,a_{\mathrm{off}}]\times[0,b_{\mathrm{off}}],\qquad
 0<a_{\mathrm{off}}<x^\ast,\qquad
 b(b_{\mathrm{off}})<d_0+c,\qquad
 \alpha_\Coop a_{\mathrm{off}}<\kappa b_{\mathrm{off}}.
\end{equation}
Such $a_{\mathrm{off}}$, $b_{\mathrm{off}}$ exist whenever
$b_0<d_0+c$, by continuity of $b$.

\begin{proposition}
\label{prop:off-target}
Let $a_{\mathrm{off}},b_{\mathrm{off}}>0$ satisfy the last two
conditions in \eqref{eq:def-R-off}. Then, for every $r>0$,
$R_{\mathrm{off}}$ is forward invariant under \eqref{eq:mean-field} and
every solution started in $R_{\mathrm{off}}$ converges to
$U_{\mathrm{off}}$. In particular
$R_{\mathrm{off}}\cap Q\subset\mathcal B_{\mathrm{off}}(r)$.
\end{proposition}

\begin{proof}
Let $\gamma:=d_0+c-b(b_{\mathrm{off}})>0$. On $R_{\mathrm{off}}$,
$\dot x=x[b(w)-d(x)]\le-\gamma x$, and on the edge $w=b_{\mathrm{off}}$,
$\dot w=r(\alpha_\Coop x-\kappa b_{\mathrm{off}})
\le r(\alpha_\Coop a_{\mathrm{off}}-\kappa b_{\mathrm{off}})<0$. Hence the
flow points into $R_{\mathrm{off}}$ on the edges $x=a_{\mathrm{off}}$ and
$w=b_{\mathrm{off}}$. Moreover, the axis $\{x=0\}$ is invariant, and on the
edge $\{w=0,\ x>0\}$ we have $\dot w=r\alpha_\Coop x>0$, i.e., the flow points
inward there as well, so $R_{\mathrm{off}}$ is forward invariant. Inside it $x(t)\le x(0)e^{-\gamma t}$, and the
variation-of-constants formula for $w$ gives, with $m:=\min\{r\kappa,\gamma\}$,
\[
 w(t)=e^{-r\kappa t}w(0)
 +r\alpha_\Coop\int_0^t e^{-r\kappa(t-s)}x(s)\,ds
 \le e^{-r\kappa t}w(0)+r\alpha_\Coop x(0)\,t\,e^{-mt}\to0 .
\]
\end{proof}

The first condition in \eqref{eq:def-R-off} keeps the target at a
positive distance from the saddle.

Starting near $U_{\mathrm{on}}$, the extinction time
\eqref{eq:collapse-time-off} is the first entry of $U^N$ into this set,
with the target fixed independently of $N$ and $r$.
Its deterministic invariance does not make it stochastically
absorbing: recovery remains possible after entry.
This event differs from both the population-threshold crossing
$X^N\le x^\ast$ and exact cell extinction, $X^N=0$.
The relation between the action required to reach this target
and the saddle action is established in
Section~\ref{sec:barrier}.

%%%%%%%%%%%%%%%%%%%%%%%%%%%%%%%%%%

\FloatBarrier
\subsection{Gaussian fluctuations}

We are now interested in the fluctuations of $U^N$ about the deterministic
solution $U$ of \eqref{eq:mean-field}. Theorem~\ref{thm:lln} makes their
difference vanish, hence we magnify it and define the fluctuation field
\begin{equation}
 \label{eq:fluct-field}
 Z^N:=\sqrt N\,\bigl(U^N-U\bigr),
\end{equation}
$\sqrt N$ being the scale at which it has a nondegenerate limit. Deviations of
that order about the deterministic path are Gaussian, provided the initial
fluctuation is.

\begin{theorem}[Kurtz; see {\cite[Ch.~11]{ethier2009markov}}]
\label{thm:clt}
Assume the hypotheses of Theorem~\ref{thm:lln} and let
$Z^N(0)\Rightarrow Z_0$ for a random vector $Z_0$ that is Gaussian or
deterministic. Then, on
the time interval of Theorem~\ref{thm:lln}, the field \eqref{eq:fluct-field}
converges weakly in the Skorokhod space $D([0,T];\R^2)$,
$Z^N\Rightarrow Z=(Z_x,Z_w)$, where $Z$ is the Gaussian process with
$Z(0)=Z_0$ solving
\begin{equation}
 \label{eq:clt-system}
 \begin{aligned}
  dZ_x&=\bigl[\bigl(b(w)-d(x)-x\,d'(x)\bigr)Z_x+x\,b'(w)\,Z_w\bigr]dt
       +\sqrt{x\,b(w)}\;dW_1-\sqrt{x\,d(x)}\;dW_2,\\
  dZ_w&=r\bigl[\alpha_\Coop Z_x-\kappa Z_w\bigr]dt
       +\sqrt{r\alpha_\Coop x}\;dW_3-\sqrt{r\kappa w}\;dW_4,
 \end{aligned}
\end{equation}
where $(x,w)=U(t)$ is the deterministic solution of \eqref{eq:mean-field} and
$W_1,\dots,W_4$ are independent standard Brownian motions, one per channel of
Figure~\ref{fig:network}, independent of $Z_0$.
\end{theorem}

The cited theorem asks one more derivative of the rate densities and a moment
condition of the jumps, both available since every $\beta_\varrho$ is smooth on
$[0,\infty)^2$ and the increments have unit length. Collecting the four noise
terms gives the diffusion matrix
$D=\sum_\varrho\Delta_\varrho\Delta_\varrho^{\!\top}\beta_\varrho$, in which,
unlike the drift, the increments enter quadratically. The channels of
Figure~\ref{fig:network} are axis-aligned, hence
\begin{equation}
 \label{eq:diffusion}
 D(x,w)=\operatorname{diag}
 \bigl(xb(w)+xd(x),\;r[\alpha_\Coop x+\kappa w]\bigr).
\end{equation}
At an asymptotically stable equilibrium the stationary covariance $\Sigma$ of
$Z$ solves the Lyapunov equation $J_r\Sigma+\Sigma J_r^{\!\top}+D=0$. At
$U_{\mathrm{on}}$ the variance of the signal conditional on the cells is
$1.88\,\bar w(x_{\mathrm{on}})$ at $r=0.5$, $2.01\,\bar w(x_{\mathrm{on}})$ at
$r=1$ and $1.02\,\bar w(x_{\mathrm{on}})$ at $r=256$, i.e., this regime recovers
the conditional Poisson equilibrium at large $r$ and departs from it by a factor
near two for $r\le2$.

\subsection{Large deviations}

By means of the exponential change of variable applied to \eqref{eq:generator},
i.e., of the nonlinear generator $N^{-1}e^{-Nf}\mathfrak L_N(e^{Nf})$ of the
Feng-Kurtz construction~\cite{feng2006large}, we obtain in the limit the
Hamiltonian
\[
 \mathbb H(U,p)
 =\sum_{\varrho}\beta_\varrho(U)
  \bigl(e^{\inner{p}{\Delta_\varrho}}-1\bigr),
\]
which for the four channels of Figure~\ref{fig:network} and momenta $(p,q)$
reads
\begin{equation}
 \label{eq:ham}
 \mathbb H(x,w,p,q)
 =\underbrace{xb(w)(e^p-1)+xd(x)(e^{-p}-1)}_{\mathbb H_s}
 +r\underbrace{\bigl[\alpha_\Coop x(e^q-1)+\kappa w(e^{-q}-1)\bigr]}_{\mathbb H_f}.
\end{equation}
Its Legendre-Fenchel transform
$\mathbb L(U,v)=\sup_p[\inner{p}{v}-\mathbb H(U,p)]$ is the local rate, and the
action of an absolutely continuous path is
$\mathcal S_T(U)=\int_0^T\mathbb L(U(t),\dot U(t))\,dt$. Deviations of order one
are exponentially unlikely at speed $N$.

\begin{theorem}[Agazzi et al.~{\cite[Thm.~1 and Rem.~2.6]{agazzi2022large}}; see also~\cite{shwartz1995large}]
\label{thm:ldp}
Let $G$ be an open set whose closure is a compact subset of $Q=(0,\infty)^2$,
and let the initial states be deterministic, $U^N(0)=u_N\to u_0\in G$. Let
$\widetilde U^N$ be the process with the jumps of Figure~\ref{fig:network} and
rate densities $\widetilde\beta_\varrho$ that coincide with $\beta_\varrho$ on
$\overline G$ and are Lipschitz, bounded and bounded away from zero on $\R^2$,
started at $\widetilde U^N(0)=u_N$.
Then $\widetilde U^N$ satisfies a large-deviation principle on $[0,T]$ at speed $N$
with the good rate function
\[
 I_{u_0}(U)=
 \begin{cases}
  \widetilde{\mathcal S}_T(U), & U\in\mathrm{AC}([0,T];\R^2),\ U(0)=u_0,\\
  +\infty, & \text{otherwise},
 \end{cases}
\]
where $\widetilde{\mathcal S}_T$ is the action built from $\widetilde\beta_\varrho$.
That is, the sublevel sets of $I_{u_0}$ are compact, and for every closed set
$C$ and every open set $O$ of the Skorokhod space $D([0,T];\R^2)$ with the $J_1$ topology,
\[
 \begin{aligned}
  \limsup_{N\to\infty}\frac1N\log\mathbb P\bigl(\widetilde U^N\in C\bigr)
   &\le-\inf_{U\in C}I_{u_0}(U),\\
  \liminf_{N\to\infty}\frac1N\log\mathbb P\bigl(\widetilde U^N\in O\bigr)
   &\ge-\inf_{U\in O}I_{u_0}(U).
 \end{aligned}
\]
\end{theorem}

Such extensions exist, e.g. the rate densities composed with the projection onto
a compact rectangle in $Q$ containing $\overline G$. The processes can be coupled
up to their first exit from $G$. Consequently, the displayed bounds also hold for
$U^N$ for open or closed path events all of whose paths remain in $G$, with
$\widetilde{\mathcal S}_T$ replaced by $\mathcal S_T$. This is the only
localization used below. Results of this type go back
to~\cite{leonard1995large,shwartz1995large}, and the flux formulation is given
in~\cite{patterson2019large}. The localization is needed because the rate
densities vanish on the axes. The cited theorem also covers such degenerate
sets, under an escape condition on the rates near them that we do not verify
for this network. The variational results of Section~\ref{sec:barrier} restrict
every admissible path to the open quadrant. The local formulation above does
not by itself give an exit-time theorem for the original chain, nor a
comparison with paths that touch the axes.

\begin{remark}
\label{rem:H-contains}
The third regime contains the other two, in the sense that they are the first
two terms of $\mathbb H$ at zero momentum. Differentiating
$\mathbb H(U,p)=\sum_\varrho\beta_\varrho(U)(e^{\inner{p}{\Delta_\varrho}}-1)$
once and twice at $p=0$ gives $\sum_\varrho\beta_\varrho\Delta_\varrho=F(U)$
and $\sum_\varrho\beta_\varrho\Delta_\varrho\Delta_\varrho^{\!\top}=D(U)$,
i.e., the drift of \eqref{eq:mean-field} and the diffusion matrix
\eqref{eq:diffusion}, hence
\[
 \begin{aligned}
  \mathbb H(U,p)&=\inner{F(U)}{p}+\tfrac12\inner{p}{D(U)p}
   +O(\lvert p\rvert^3),\\
  \mathbb L(U,v)&=\tfrac12\inner{v-F(U)}{D(U)^{-1}\bigl(v-F(U)\bigr)}
   +O(\lvert v-F(U)\rvert^3).
 \end{aligned}
\]
The cost of a velocity is therefore measured, to leading order, in the norm
that the mobility of Theorem~\ref{thm:clt} supplies, and to that order the
action is the one of the diffusion driven by the noise of
\eqref{eq:clt-system}. Notice that this is also why the calibration of
Section~\ref{sec:mam} works: near $U_{\mathrm{on}}$ the quasipotential is
governed by the quadratic form above, whose metric is $D^{-1}$. The terms
of third and higher order are where the increments enter exponentially rather
than averaged as in $F$ or quadratically as in $D$. At the parameters of this
paper they are a small correction to the barrier: on the eliminated-signal
instanton $\lvert p\rvert\le0.131$, the barrier of the quadratic Hamiltonian
differs from $\Delta V_\infty$ by $0.1\%$, and Section~\ref{sec:results} gives
the same comparison at finite $r$. Two networks with
the same drift and different channels therefore share
Theorem~\ref{thm:lln} and can differ in both later regimes. The drift carries
net rates and the diffusion gross ones: at $U_{\mathrm{on}}$ the cell drift
vanishes because $b=d$, whereas the cell diffusion is $2x\,b(w)$. The factor
$r$ is a case in point. By Proposition~\ref{prop:r-independence} it leaves the
equilibria and their stability types unchanged, but it survives in $J_r$ and in
$D_{ww}$, hence in the stationary covariance $\Sigma$, and in the barrier.
Since the rate densities vanish simultaneously only at the origin,
$U_{\mathrm{off}}$ is the only absorbing state of $U^N$, and the cooperative
state is metastable for the jump process: its location is common to every $r$,
its spread $\Sigma/N$ depends on $r$, and, subject to the exit estimates
discussed after \eqref{eq:lifetime-prediction}, its lifetime is of order
$\exp\{N\Delta V(r)\}$.
\end{remark}

%=====================================================================
\section{The extinction barrier}
\label{sec:barrier}

The extinction time \eqref{eq:collapse-time-off} is the first entry into the target
$R_{\mathrm{off}}$ of Section~\ref{sec:off-target}. This section relates the
least action needed to reach that target to the least action needed to reach
the saddle $U_\ast$, which is the quantity we compute.
Theorem~\ref{thm:ldp} assigns the action $\mathcal S_T(U)$ to an absolutely
continuous path $U$ with values in a compact subset of $Q$. On the exponential
scale, a small neighbourhood of such a path has probability of order
$\exp\{-N\mathcal S_T(U)\}$, so the likeliest ways of reaching a target are
paths of nearly minimal action. Accordingly, we restrict every admissible path
to the open quadrant and define, for $u,z\in Q$,
\[
 V_r^Q(u,z)
 :=\inf_{T>0}\;
  \inf_{\substack{U\in\mathrm{AC}([0,T];Q):\\
      U(0)=u,\;U(T)=z}}
  \mathcal S_T(U),
 \qquad
 V_r^Q(z):=V_r^Q(U_{\mathrm{on}},z).
\]
The superscript records the restriction to $Q$. The quasipotential barrier is
the value at the saddle
\begin{equation}
 \label{eq:target}
 \Delta V(r)
 :=V_r^Q(U_\ast)
 =\inf_{T>0}\;
  \inf_{\substack{U\in\mathrm{AC}([0,T];Q):\\
      U(0)=U_{\mathrm{on}},\;U(T)=U_\ast}}
  \mathcal S_T(U).
\end{equation}
Here the first infimum allows the path to choose its duration. This is needed
because the minimizing orbit approaches the two fixed points only
asymptotically. Hence no finite value of $T$ realizes the free-time minimum;
instead, the minimum is approached as $T\to\infty$. This is the reason for the
horizon study in Section~\ref{sec:mam}. The stored minimizers at $c=0.36$,
$r\in\{0.5,1,4,16\}$, stay in $[x^\ast,\infty)^2$, well inside $Q$.

Equation~\eqref{eq:target} is a genuine minimization problem. The local rate
is nonnegative and vanishes along the drift, i.e., $\mathbb L(U,F(U))=0$,
since $\mathbb H(U,0)=0$, $p\mapsto\mathbb H(U,p)$ is convex and
$\nabla_p\mathbb H(U,0)=F(U)$ (Remark~\ref{rem:H-contains}); hence following
the deterministic drift costs nothing. Leaving the basin of the
cooperative equilibrium requires a fluctuation; after entering the off basin,
the deterministic evolution can carry the system towards the off state at zero
action. Moreover, at finite $r$ the state is two-dimensional: the signal need
not remain at $w=\bar w(x)$. The cheapest path can exploit signal lag, as
illustrated in Figure~\ref{fig:mechanism}. Since the network is not
reversible, there is also no general reason for this path to be the
deterministic relaxation path run backwards.

The saddle enters as in Freidlin-Wentzell theory~\cite{freidlin2012random}. A
path from $U_{\mathrm{on}}$ into the off basin must cross the separatrix, and
from any point of the separatrix the deterministic flow leads to $U_\ast$ at
zero action, so no point of the separatrix is cheaper to reach than $U_\ast$.
Conversely, from $U_\ast$ the branch of the unstable manifold that leaves
towards smaller $x$ converges to $U_{\mathrm{off}}$, hence enters
$R_{\mathrm{off}}$, at zero action.

\begin{proposition}
\label{prop:saddle-off}
Let $c\in[0.2,0.5]$ and $r>0$, and let $R\subset[0,\infty)^2$ be a relative
neighbourhood of $U_{\mathrm{off}}$ with
$R\cap Q\subset\mathcal B_{\mathrm{off}}(r)$. Then
$\inf_{z\in R\cap Q}V_r^Q(z)=\Delta V(r)$. In particular, by
Proposition~\ref{prop:off-target},
\begin{equation}
 \label{eq:saddle-off}
 \inf_{z\in R_{\mathrm{off}}\cap Q}V_r^Q(z)=\Delta V(r).
\end{equation}
\end{proposition}

The proof is given in Appendix~\ref{app:action}. Its lower-bound step uses
only that the endpoint lies outside $\mathcal B_{\mathrm{on}}(r)$, hence it
also gives $V_r^Q(z)\ge\Delta V(r)$ for every
$z\in Q\setminus\mathcal B_{\mathrm{on}}(r)$. The proposition concerns an
infimum of action and does not compare individual paths. Reaching $U_\ast$
does not force extinction, since noise near the saddle can send the process
to either basin.

The identification \eqref{eq:saddle-off} concerns actions. Relating it to
the mean of the extinction time \eqref{eq:collapse-time-off} additionally
requires metastable exit estimates for the jump process, appropriate to the
target $R_{\mathrm{off}}$ and to the vanishing of the rates on the coordinate
axes; we do not derive them here. Subject to such estimates, the prediction is
\begin{equation}
 \label{eq:lifetime-prediction}
 \lim_{N\to\infty}\frac1N\log\E\bigl[\tau_N^{\mathrm{off}}(r)\bigr]
 =\Delta V(r).
\end{equation}
In that case a change $\delta$ in the barrier multiplies the mean lifetime by
$e^{N\delta+o(N)}$, so even a modest change in $\Delta V$ has a large effect
when $N$ is large. Section~\ref{sec:ssa} states which stopping times are
computed.

The extinction time \eqref{eq:collapse-time-off} counts every failed attempt
from the saddle region. If successive attempts were independent, the number of
attempts would be geometric with success probability $p_N$, and the mean
extinction time would be the mean duration of an attempt divided by $p_N$.
A success probability with $\log p_N=o(N)$ would then change the mean
extinction time only by a subexponential factor, leaving the exponent in
\eqref{eq:lifetime-prediction} unchanged. We do not prove that the jump process
behaves in this way; Section~\ref{sec:ssa} reports the success probabilities
observed in simulation.

\subsection{Barrier after eliminating the signal}
\label{ssec:eliminated}

Eliminating the signal removes the extra freedom in $w$. By~\eqref{eq:slaved}, it imposes
$w=\bar w(x)=\alpha_\Coop x/\kappa$ and leaves a one-dimensional birth-death
process for $x$, with birth rate $x\,\bar b(x)$ and death rate $x\,d(x)$. The
eliminated-signal Hamiltonian is then
\[
 \mathbb H_\infty(x,p)
 =x\bigl[\bar b(x)(e^p-1)+d(x)(e^{-p}-1)\bigr],
\]
i.e., the full Hamiltonian~\eqref{eq:ham} at $w=\bar w(x)$ and $q=0$, where
$\mathbb H_f$ vanishes. Notice that the signal itself has not been switched
off: its quasi-steady value still enters the cell division rate through
$\bar b(x)=b(\bar w(x))$. We denote the barrier of this model by
$\Delta V_\infty$. This is a definition, not an interchange of limits: the
relation $w=\bar w(x)$ is the limit $r\to\infty$ taken after $N\to\infty$, and
whether $\Delta V_\infty=\lim_{r\to\infty}\Delta V(r)$ is a separate question,
left for~\cite{ayala2026cooperation}.

The advantage of the eliminated-signal model is that its barrier can be calculated
explicitly as follows. Along a free-time minimizing orbit the Hamiltonian is conserved
and has value zero. Setting $z=e^p>0$, from $\mathbb H_\infty(x,p)=0$ and $x>0$ we obtain
\[
 0=\bar b(x)(z-1)+d(x)(z^{-1}-1),
 \qquad
 0=(z-1)\bigl(\bar b(x)z-d(x)\bigr).
\]
There are therefore two zero-energy branches
\begin{equation}
 \label{eq:two-branches}
 p=0,
 \qquad
 p=p^\ast(x):=\log\frac{d(x)}{\bar b(x)}.
\end{equation}
Their directions follow from Hamilton's equation
$\dot x=\partial_p\mathbb H_\infty$
\[
 \left.\dot x\right|_{p=0}=x\bigl(\bar b(x)-d(x)\bigr)>0,
 \qquad
 \left.\dot x\right|_{p=p^\ast}=x\bigl(d(x)-\bar b(x)\bigr)<0
 \quad\text{for }x\in(x^\ast,x_{\mathrm{on}}).
\]
The first branch is the deterministic motion back toward the cooperative
state. The second branch moves toward the threshold and is the instanton of
the eliminated-signal model.

Because $\mathbb H_\infty=0$ on this branch, its action reduces to the line
integral of the momentum
\begin{equation}
 \label{eq:Vinf}
 \begin{aligned}
 \Delta V_\infty
 &=\int p^\ast(x)\,dx
   =\int_{x_{\mathrm{on}}}^{x^\ast}
       \log\frac{d(x)}{\bar b(x)}\,dx \\
 &=\int_{x^\ast}^{x_{\mathrm{on}}}
       \log\frac{b(\bar w(x))}{d(x)}\,dx .
 \end{aligned}
\end{equation}
Notice that on $(x^\ast,x_{\mathrm{on}})$ we have $\bar b(x)>d(x)$, hence the last
integrand is positive, i.e., the instanton moves against the deterministic drift. At $c=0.36$,
\[
 \Delta V_\infty=0.06138.
\]

\subsection{What changes at finite \texorpdfstring{$r$}{r}}
\label{ssec:finite-r}

For finite $r$, the zero-energy condition is
\[
 \mathbb H(x,w,p,q)=0.
\]
This is one equation for the two momenta $p$ and $q$, so it no longer selects a
unique momentum as a function of the state. Where the quasipotential is
differentiable it still satisfies $\mathbb H(U,\nabla V_r^Q(U))=0$, and along a
smooth minimizing path $(p,q)=\nabla V_r^Q$; but this Hamilton-Jacobi equation
in two variables does not reduce to a choice between branches as in
\eqref{eq:two-branches}. The cell and signal coordinates also remain coupled
through the coefficients of~\eqref{eq:ham}. These are the reasons why no
one-dimensional quadrature analogous to~\eqref{eq:Vinf} is available and why
the finite-$r$ problem must be minimized numerically in Section~\ref{sec:mam}.
The network is not reversible, which is why the minimizing path need not be the
deterministic relaxation path run backwards (Section~\ref{sec:barrier}); it
does not by itself prevent a potential.

Nevertheless, the eliminated-signal instanton gives a useful comparison path for large
$r$. Since $\mathbb H_s$ contains only $p$ and $\mathbb H_f$ only $q$, the
Legendre transform separates
\begin{equation*}
 \mathbb L(U,v)
 =\mathbb L_s(U,v_1)
  +r\,\mathbb L_f\!\left(U,\frac{v_2}{r}\right).
\end{equation*}
Let $x(t)$ be the eliminated-signal instanton,
\begin{equation}
 \label{eq:elim-instanton}
 \dot x=x\,[d(x)-\bar b(x)],
\end{equation}
and lift it to the plane by setting $U_{\mathrm{sl}}(t)=(x(t),\bar w(x(t)))$.
Its cell contribution is exactly $\Delta V_\infty$; the remaining one is the
cost of making the signal follow the slaved relation at a nonzero speed. Along
this path the signal production and removal rates are equal,
$a:=\alpha_\Coop x=\kappa\bar w(x)$, hence $\mathbb H_f=2a(\cosh q-1)$ and
$\mathbb L_f(U,u)=u^2/(4a)+O(u^4)$. With $\dot w=(\alpha_\Coop/\kappa)\dot x$
and, by \eqref{eq:elim-instanton},
$\int\dot x^2/x\,dt=\int_{x^\ast}^{x_{\mathrm{on}}}[\bar b(x)-d(x)]\,dx$, we
obtain
\[
 r\mathbb L_f\!\left(U,\frac{\dot w}{r}\right)
 =\frac1r\frac{\alpha_\Coop\dot x^2}{4\kappa^2x}+O(r^{-3}).
\]
Hence the lifted path has action
\begin{equation}
 \label{eq:slaved-upper-bound}
 \mathcal S(U_{\mathrm{sl}})
 =\Delta V_\infty+\frac{C_{\mathrm{sl}}}{r}+O(r^{-3}),
 \qquad
 C_{\mathrm{sl}}
 =\frac{\alpha_\Coop}{4\kappa^2}
  \int_{x^\ast}^{x_{\mathrm{on}}}
    \bigl[\bar b(x)-d(x)\bigr]\,dx.
\end{equation}
The instanton reaches the fixed points only at infinite times. Truncating it
near both endpoints and then letting the endpoint corrections shrink gives an
admissible sequence for~\eqref{eq:target}, so~\eqref{eq:slaved-upper-bound}
implies
\begin{equation}
 \label{eq:barrier-upper-bound}
 \Delta V(r)
 \le \Delta V_\infty+\frac{C_{\mathrm{sl}}}{r}+O(r^{-3}).
\end{equation}
At $c=0.36$, $C_{\mathrm{sl}}=0.06579$.

The upper bound keeps the signal on the slaved relation~\eqref{eq:slaved}. Let us now allow a
displacement of order $1/r$, i.e., $w=\bar w(x)+\xi/r$. Then the signal drift
is $-\kappa\xi$, the cell Lagrangian changes by $(\xi/r)\,m$ with
$m:=\partial_w\mathbb L_s\bigl((x,\bar w(x)),\dot x\bigr)$, and expanding
$\mathbb L_f$ about its drift as above we obtain, formally,
\begin{equation}
 \label{eq:first-order-action}
 \mathcal S(U)
 =\int\mathbb L_s\bigl((x,\bar w(x)),\dot x\bigr)\,dt
 +\frac1r\int\left[\xi\,m
  +\frac{\bigl(\bar w'(x)\dot x+\kappa\xi\bigr)^2}{4\alpha_\Coop x}\right]dt
 +O(r^{-2}).
\end{equation}
Notice that $\dot\xi$ does not occur, hence the second integral is minimized
pointwise in $\xi$, and to this order it is evaluated on the
instanton~\eqref{eq:elim-instanton}, where
$m=b'(\bar w(x))\bar F(x)/\bar b(x)$ by the envelope theorem. The choice
$\xi=0$ returns $C_{\mathrm{sl}}$. A signal which follows its own drift, i.e.,
$\kappa\xi=-\bar w'(x)\dot x$, pays no signal action and gives the coefficient
$A_{\mathrm{lag}}=(\alpha_\Coop/\kappa^2)\int_{x^\ast}^{x_{\mathrm{on}}}m\,dx$.
The minimizer gives
\begin{equation}
 \label{eq:A-theory}
 A=\frac{\alpha_\Coop}{\kappa^2}
 \int_{x^\ast}^{x_{\mathrm{on}}}\left[m-\frac{x\,m^2}{\bar F(x)}\right]dx
 =\frac{\alpha_\Coop}{\kappa^2}
 \int_{x^\ast}^{x_{\mathrm{on}}}
 \frac{b'(\bar w(x))\bar F(x)}{b(\bar w(x))}
 \left[1-\frac{x b'(\bar w(x))}{b(\bar w(x))}\right]dx ,
\end{equation}
i.e., the lag of the signal raises the barrier by $A_{\mathrm{lag}}/r$, and the
fluctuations of the signal remove the second term. At $c=0.36$ we have
$C_{\mathrm{sl}}=0.0658$, $A_{\mathrm{lag}}=0.0472$ and $A=0.0354$. If the signal
count is scaled as $\nu N$, i.e., the signal jumps have size $1/(\nu N)$ and
their rates are multiplied by $\nu$, then the quadratic term
in~\eqref{eq:first-order-action} is multiplied by $\nu$ and the same
minimization gives
\begin{equation}
 \label{eq:A-nu}
 A(\nu)=A_{\mathrm{lag}}-\frac{A_{\mathrm{lag}}-A}{\nu},
\end{equation}
which interpolates between the present model, $\nu=1$, and a signal without
fluctuations, $\nu\to\infty$.

This calculation is formal. It assumes that the minimizing paths have the form
$w=\bar w(x)+\xi/r$ uniformly in time, and it exchanges the minimization with
the expansion; the details are given in~\cite{ayala2026cooperation}. Here we
use~\eqref{eq:A-theory} and~\eqref{eq:A-nu} as predictions, which are tested in
Section~\ref{sec:results}. The comparison path proves only the upper
bound~\eqref{eq:barrier-upper-bound}. Neither it nor the formal calculation
proves rigorously that $\Delta V(r)\ge\Delta V_\infty$, nor that the leading correction is
$A/r$.

%=====================================================================
\section{Computing the barrier: full case}
\label{sec:numerics}

We develop the two routes in turn.

\subsection{Minimum action}
\label{sec:mam}

\subsubsection*{Discretization}

We compute \eqref{eq:target} by a minimum action method~\cite{e2004minimum}. For a horizon $T$ and a uniform mesh $t_k=k\Delta t$ with $\Delta t=T/M$, we
represent a path by its nodes $\boldsymbol U=(U_0,\ldots,U_M)$ with
$U_0=U_{\mathrm{on}}$ and $U_M=U_\ast$ held fixed, and approximate its action
by
\begin{equation}
 \label{eq:discrete-action}
 \mathcal S_{T,M}^{(r)}(\boldsymbol U)
 :=\sum_{k=0}^{M-1}\Delta t\,
 \mathbb L^{(r)}
 \left(\frac{U_k+U_{k+1}}{2},\;\frac{U_{k+1}-U_k}{\Delta t}\right),
\end{equation}
i.e., a midpoint rule in the state variable with forward-difference velocities.
The interior nodes $U_1,\ldots,U_{M-1}$ are the optimization variables.
Since the endpoint $U_M=U_\ast$ is imposed, the method approximates the
saddle barrier \eqref{eq:target}. It does not select an endpoint, and in
particular it does not locate the cheapest point of the threshold line
$\{x=x^\ast\}$.

\subsubsection*{Evaluating the Lagrangian}

The implementation evaluates the Legendre transform pointwise. For a pair of
opposite unit jumps with positive intensities $\lambda_+$ and $\lambda_-$,
write
\begin{equation}
 \label{eq:explicit-legendre}
 p^\ast(v;\lambda_+,\lambda_-)
 =\log\frac{v+\sqrt{v^2+4\lambda_+\lambda_-}}{2\lambda_+},
 \qquad
 \ell(v;\lambda_+,\lambda_-)
 =vp^\ast-\sqrt{v^2+4\lambda_+\lambda_-}
             +\lambda_+ +\lambda_- .
\end{equation}
The reaction directions in Figure~\ref{fig:network} are axis-aligned, hence the
local rate is explicitly
\begin{equation}
 \label{eq:explicit-L}
 \mathbb L(U,v)=
 \ell\bigl(v_x;xb(w),xd(x)\bigr)
 +\ell\bigl(v_w;r\alpha_\Coop x,r\kappa w\bigr),
\end{equation}
where the two intensities are passed as the second and third arguments of
$\ell$. This closed form is special to the present network.

\subsubsection*{Implementation}

For transfer to networks with coupled jump directions, the production code
computes the maximizing momentum by a generic damped Newton solve of
\[
 v_k=\partial_p\mathbb H^{(r)}(U_k,p_k),
 \qquad
 \mathbb L^{(r)}(U_k,v_k)
 =\inner{p_k}{v_k}-\mathbb H^{(r)}(U_k,p_k),
\]
which agrees with \eqref{eq:explicit-legendre} and \eqref{eq:explicit-L} to
$4\times10^{-15}$ on $2000$ random state-velocity pairs. By the envelope
theorem~\cite{rockafellar1970convex}, $\partial_v\mathbb L^{(r)}=p^\ast$ and
$\partial_U\mathbb L^{(r)}=-\partial_U\mathbb H^{(r)}$ at that momentum, so
gradients need no further Legendre solves. The interior nodes are optimized by
L-BFGS-B~\cite{byrd1995limited,zhu1997algorithm} in the implementation
of~\cite{virtanen2020scipy}, with \texttt{ftol}$=10^{-14}$,
\texttt{gtol}$=10^{-10}$ and trial states kept above $10^{-9}$, starting from
the straight segment from $U_{\mathrm{on}}$ to $U_\ast$.

\subsubsection*{Calibration and convergence}

As an end-to-end calibration, we apply the same discretization and optimizer to
the one-dimensional eliminated-signal problem. With the production settings
$T=40$, $M=2000$, its action is $0.0613801$, against
$\Delta V_\infty=0.0613795$ from \eqref{eq:Vinf}, a relative error of
$0.0010\%$.

That calibration exercises the one-dimensional eliminated-signal problem only. A
second one exercises the two-coordinate minimizer itself, at finite $r$,
against a closed form the solver never uses. Near an asymptotically stable
equilibrium the rate functional of Theorem~\ref{thm:ldp} gives a quadratic
quasipotential, i.e.,
$V(U)=\tfrac12\inner{U-U_{\mathrm{on}}}{\Pi(U-U_{\mathrm{on}})}
+O(\lvert U-U_{\mathrm{on}}\rvert^3)$, whose Hessian $\Pi$ solves
$\Pi J_r+J_r^{\!\top}\Pi+\Pi D\Pi=0$, with $D$ as in \eqref{eq:diffusion}. Multiplying that identity by $\Pi^{-1}$ on both sides
returns the Lyapunov equation of Theorem~\ref{thm:clt}, hence
\begin{equation}
 \label{eq:hessian}
 \Pi=\Sigma^{-1},
\end{equation}
i.e., the quasipotential Hessian is the inverse stationary covariance of the
Gaussian regime, available from a $2\times2$ Lyapunov solve. The identity
\eqref{eq:hessian} does not presume reversibility, which the network lacks
(Kolmogorov's criterion fails around every unit cell of the lattice): it is an
algebraic consequence of the Riccati and Lyapunov equations above. We therefore
minimize the action from $U_{\mathrm{on}}$ to $U_{\mathrm{on}}+\varepsilon z$
over six unit directions $z$ and two excursion sizes, extrapolate the quadratic
forms $2\varepsilon^{-2}\mathcal S$ to $\varepsilon\to0$, and obtain $\Pi$ by
least squares. Its three independent entries agree with \eqref{eq:hessian} to a
relative error of at most $2.5\times10^{-4}$ over $r\in\{0.5,1,2,4,8\}$.
Notice that this check shares nothing with the one-dimensional calibration and
nothing with the exit-time experiment of Section~\ref{sec:ssa}: it tests the
finite-$r$ two-coordinate minimizer against Theorem~\ref{thm:clt}. It is
nevertheless local, hence it constrains the solver near $U_{\mathrm{on}}$ and
says nothing about the global minimum in \eqref{eq:target}.

We separate horizon and mesh effects at the smallest and largest values of
$r$. At fixed $\Delta t=0.02$ the $r=0.5$ barrier changes by
$7.3\times10^{-5}$ between $T=20$ and $T=30$ and by at most $10^{-6}$ up to
$T=60$, while $\Delta V(64)$ does not change; at fixed $T=40$, refining
$\Delta t$ from $0.08$ to $0.01$ changes $\Delta V(0.5)$ by $3\times10^{-6}$ and
$\Delta V(64)$ by $10^{-5}$, and the last halving by at most $10^{-6}$. Both
studies bracket the production setting $T=40$, $M=2000$.

The production sweep uses $T=40$, $M=2000$ and an outer iteration budget of
$60000$. The default budget of $600$ iterations fails upward, returning
$0.128412$ at $r=0.5$ instead of $0.122477$. We compute $r\in\{0.5,1,2,4,8,16,32,64\}$ for
$c\in\{0.20,0.30,0.36,0.45,0.50\}$, in $234$ seconds for all $40$
minimizations; code, settings and barriers are in the accompanying
repository~\cite{qsmmsccode}.

The same discretization and optimizer are used for further sweeps at
$c=0.36$. For the slow-signal regime we compute nine rates between
$r=0.354$ and $r=0.0078$ with the horizon $T=20/r$, since the signal relaxes on
the time $1/(r\kappa)$. Replacing it by $T=30/r$ changes these barriers by at
most $6\times10^{-8}$, refining $\Delta t$ from $0.08$ to $0.02$ by at most
$3\times10^{-6}$, and fresh starts from paths at $r=0.125$ and
$r=0.0625$ agree to $1.2\times10^{-6}$. For the comparison with the diffusion
approximation we replace $\mathbb L$ by its quadratic part
$\tfrac12\inner{v-F}{D^{-1}(v-F)}$ of Remark~\ref{rem:H-contains}; on the
eliminated-signal problem this solver returns $0.0613218$ against the closed
form $\int_{x^\ast}^{x_{\mathrm{on}}}2(\bar b-d)/(\bar b+d)\,dx=0.0613211$.
Moreover, for a signal count scaled as $\nu N$ the signal term
of~\eqref{eq:explicit-L} is multiplied by $\nu$, and we compute
$\nu\in\{1,2,4,16,64\}$ at $r\in\{8,16,32,64\}$.

Three further checks address the global and the discrete character of the
minimum at $c=0.36$ and $r\in\{0.5,1,4,16,64\}$. First, besides the straight
segment we start from the same segment bowed up and down in $w$ and from the
lifted eliminated-signal instanton of Section~\ref{sec:barrier}. All four
starts converge to the same path, within Hausdorff distance $4\times10^{-3}$,
and to the same action within $2.2\times10^{-7}$. Second, a free-time
minimizer has $\mathbb H=0$ along the path, and the computed paths give
$\max_k\lvert\mathbb H(U_k,p_k)\rvert\le9\times10^{-7}$. Third, the midpoint
sum \eqref{eq:discrete-action} is not itself an upper bound for the infimum,
whereas the exact action of any admissible path is. The exact action of the
piecewise-linear interpolant of each computed path, integrated by Gauss-Legendre
quadrature on every interval, lies $7\times10^{-8}$ to $4\times10^{-7}$ below
the midpoint value, hence the reported barriers are upper bounds for
\eqref{eq:target}, and none of the starts reveals a lower minimum.

\subsubsection*{Extraction of the coefficient}

For each displayed value of $c$ we extract the coefficient
$A_{\mathrm{num}}(c)$ from the large-$r$ numerical barriers through
\begin{equation}
 \label{eq:fit}
 \Delta V_{\mathrm{num}}(r;c)
 =\Delta V_\infty(c)+\frac{A_{\mathrm{num}}(c)}{r}+O(r^{-2}),
\end{equation}
with $\Delta V_\infty(c)$ fixed at its quadrature value, i.e.,
$A_{\mathrm{num}}(c)$ is the plateau of
$r[\Delta V_{\mathrm{num}}(r;c)-\Delta V_\infty(c)]$ at large $r$. We read it
at $r=64$ and use the difference from $r=32$ as a residual estimate. We
compare it with the prediction~\eqref{eq:A-theory}, evaluated by quadrature at
each cost; the formal calculation behind it is not used in the minimization.

\subsection{Exit times}
\label{sec:ssa}

The minimum-action calculation and the expansion \eqref{eq:crossover-law} share
the same variational structure, hence agreement between them would not be an
independent test. We therefore compute mean exit times directly from the
generator~\eqref{eq:generator}, by solving its backward equation on a truncated
lattice and by simulating the jump process. The primary stopping time is the
extinction time \eqref{eq:collapse-time-off}. We also compute the first time the
cell density reaches the threshold,
\begin{equation}
 \label{eq:threshold-time}
 \tau_N^{\mathrm{thr}}(r):=\inf\{t\ge0:X_t^N\le x^\ast\},
\end{equation}
where $X_t^N$ denotes the cell density. For a process started above the
threshold, $\tau_N^{\mathrm{thr}}(r)\le\tau_N^{\mathrm{off}}(r)$, since
$a_{\mathrm{off}}<x^\ast$ and the density moves by jumps of size $1/N$.
Assuming metastable exit asymptotics~\cite{freidlin2012random}, we write, for
either stopping time,
\begin{equation}
 \label{eq:exit-asymptotic}
 \E[\tau_N(r)]=C_r(N)\exp\{N\Delta V_\tau(r)\},
 \qquad
 \frac{\log C_r(N)}{N}\to0,
\end{equation}
where $\Delta V_\tau=\Delta V$ for the extinction time by
\eqref{eq:lifetime-prediction} and $\Delta V_\tau=\Delta V_{\mathrm{thr}}$, defined
below, for the threshold time. We do not prove \eqref{eq:exit-asymptotic} for
this jump process, whose rates vanish on the axes; the computations below test
it for $40\le N\le400$. For exit through an interior saddle the prefactor
$C_r(N)$ is expected to tend to a constant, and this is indeed what we find. Notice
that neither route constructs an instanton, minimizes an action, or uses a
boundary-layer ansatz.

\subsubsection*{Backward equation}

The mean $\tau(u):=\E_u[\tau_N]$ of either stopping time solves the backward
equation $\mathfrak L_N\tau=-1$ outside the target, with $\tau=0$ on the target. We
solve it at $c=0.36$ on the lattice truncated at $U_{\mathrm{on}}$ plus nine
stationary standard deviations of $\Sigma/N$ in each coordinate, where the jumps
out of this box are suppressed. Since $\tau$ is of order
$e^{N\Delta V}$, a direct solve loses accuracy as $N$ grows, hence we use a
renewal identity. Let $u_0$ be the lattice point nearest to $U_{\mathrm{on}}$,
which is also the starting state of the simulations below, with total jump
rate $q_0$ and jump probabilities $P(u_0,j)$. For $u\ne u_0$ let $h(u)$ be the
probability of reaching the target before $u_0$, and $\psi(u)$ the mean time
to reach the target or $u_0$. By the strong Markov property at the returns to
$u_0$,
\begin{equation}
 \label{eq:renewal}
 \E_{u_0}[\tau_N]
 =\frac{q_0^{-1}+\sum_jP(u_0,j)\,\psi(j)}{\sum_jP(u_0,j)\,h(j)} .
\end{equation}
The functions $h$ and $\psi$ solve two sparse linear systems on the lattice
without $u_0$, which we solve directly. We use $N=40,\dots,400$, with $25$ to
$31$ sizes per rate including every simulated size, and three targets:
$R_{\mathrm{off}}$, the same rectangle with both sides halved, and the
threshold \eqref{eq:threshold-time}. These are $399$ solves with up to
$1.4\times10^6$ unknowns, in $36$ minutes of computing time on a standard laptop.

We checked the solver as follows. Enlarging the box from $9$ to $12$ standard
deviations changes $\E[\tau_N]$ by at most $4\times10^{-14}$ in relative terms,
and a second elimination ordering by at most $3\times10^{-14}$, in the $15$
cases $N\in\{40,80,120\}$, $r\in\{0.5,1,2,4,8\}$. The relative residual of the
system for $h$ is at most $9\times10^{-15}$. Where $N\Delta V\le21$ the direct
solve of $\mathfrak L_N\tau=-1$ agrees with \eqref{eq:renewal} to $10^{-8}$ for
$\tau<10^7$ and drifts to $2.5\times10^{-5}$ at the largest $\tau$.

These values carry no sampling error. For each $r$ we fit them, by least
squares over $N\ge N_{\min}$, with
\begin{equation}
 \label{eq:bke-fit}
 \log\E[\tau_N(r)]
 =a_r+N\Delta V_{\mathrm{BE}}(r)+\frac{c_r}{N}+\frac{e_r}{N^2},
\end{equation}
i.e., a constant prefactor $e^{a_r}$ with its first two corrections, and we read
the model error from the change of $\Delta V_{\mathrm{BE}}(r)$ with
$N_{\min}\in\{120,200\}$ and with $e_r=0$.

\subsubsection*{Simulation protocol and estimator}

With the help of the stochastic simulation
algorithm~\cite{gillespie1977exact,andersonkurtz2015} we simulate the four
channels of Figure~\ref{fig:network} exactly. Each replica starts from the
integer state $n^\Coop(0)=[Nx_{\mathrm{on}}]$, $n^{\mathrm w}(0)=[N\bar
w(x_{\mathrm{on}})]$, which depends on $N$ but not on $r$. We use
$r\in\{0.5,1,2,4,8\}$ at $c=0.36$, with five to nine population sizes per rate
in $40\le N\le140$, and each pair $(r,N)$ has its own reproducibly seeded random
stream. A first campaign, with $1500$ to $2500$ replicas per pair, stops at
$\tau_N^{\mathrm{thr}}$. A second campaign follows $1500$ replicas per pair until
entry into $R_{\mathrm{off}}=[0,0.3]\times[0,0.7]$, which satisfies
\eqref{eq:def-R-off} at every cost used, and records $\tau_N^{\mathrm{thr}}$ on
the way. It uses a separately written simulator, checked against five stored
points of the first campaign (combined $p=0.20$).

Replicas are run to a deterministic censoring time $T^{\mathrm c}_{r,N}$. For
replica $k$ let $t_k$ be the smaller of $\tau_{N,k}(r)$ and
$T^{\mathrm c}_{r,N}$, and let $\delta_k=1$ if
$\tau_{N,k}(r)\le T^{\mathrm c}_{r,N}$ and $\delta_k=0$ otherwise. Under the exponential
exit-time approximation, the right-censored maximum-likelihood estimator of the
escape rate is given by
\[
 \widehat\lambda_{r,N}:=\frac{\sum_k\delta_k}{\sum_k t_k},
 \qquad
 \widehat{\E[\tau_N(r)]}:=\widehat\lambda_{r,N}^{-1}.
\]
The censoring time is set from the minimum-action barrier and no prefactor,
i.e., $T^{\mathrm c}_{r,N}=\eta\,e^{N\Delta V_{\mathrm{num}}(r)}$, with $\eta=30$
for the extinction time, since $\E[\tau_N^{\mathrm{off}}]$ is $3$ to $4$ times
$\exp\{N\Delta V_{\mathrm{num}}(r)\}$ at these sizes, and $\eta=1.5$ for the
threshold time ($0.3$ at $r=0.5$ with $N\ge90$). Then $13$ of the $51\,000$
extinction times are censored, and the threshold escape fractions range from
$0.53$ to $1.00$. The action value sets the computational budget but is not
inserted into the estimated mean. For each $r$, by weighted least squares over
the simulated sizes, we fit
\begin{equation}
 \label{eq:regression}
 \log\widehat{\E[\tau_N(r)]}=a_r+N\Delta V_{\mathrm{SSA}}(r)+\frac{c_r}{N},
\end{equation}
with weights $n_{\mathrm{exit}}(r,N)$, since the variance of
$\log\widehat{\E[\tau_N(r)]}$ is approximately $n_{\mathrm{exit}}^{-1}$, and
uncertainties are inflated by $\max(1,\chi^2_{\mathrm{red}})^{1/2}$. The
term $c_r/N$ is the first correction to a constant prefactor. Since the weights are known counting
variances, the reference distribution of a standardized difference is normal
when $\chi^2_{\mathrm{red}}\le1$. We nevertheless quote two-sided $p$-values
from Student's $t$ distribution with the degrees of freedom of the fit, which
are the larger ones.

The estimator behaves as the exponential approximation requires. Over five
independent seeds at $r=0.5$, $N=60$, the between-seed standard deviation of
$\log\widehat{\E[\tau_N^{\mathrm{thr}}]}$ is $0.0236$, against the nominal
$1500^{-1/2}=0.0258$, and censoring times imposed post hoc change the censored
estimate by at most $2\%$ at escape fractions near $0.65$, with no systematic
trend. At all $34$ pairs of the second campaign the extinction times satisfy
the exponential identities $\operatorname{sd}/\E=1$ and
$\operatorname{median}/\E=\log2$ within sampling error: the ratios lie in
$0.946$--$1.042$ and $0.653$--$0.730$, and the Kolmogorov-Smirnov test against
the exponential law with the fitted mean, censored at
$T^{\mathrm c}_{r,N}$, gives
$p\ge0.035$ at every pair, which is not more extreme than expected among $34$
tests.

\subsubsection*{The extinction time}

Figure~\ref{fig:extinction-scaling} shows the prefactor
$\E[\tau_N^{\mathrm{off}}]\,e^{-N\Delta V_{\mathrm{num}}(r)}$ of the backward
equation against $1/N$. It decreases monotonically, from $3.3$--$4.2$ at $N=40$
to $2.5$--$3.3$ at $N=400$, and it is linear in $1/N$ at large $N$, i.e.,
$C_r(N)$ tends to a constant with a correction of relative order $1/N$. The fit
\eqref{eq:bke-fit} gives the growth rates $\Delta V_{\mathrm{BE}}(r)$ of
Table~\ref{tab:exit-barriers}. They differ from $\Delta V_{\mathrm{num}}(r)$ by at
most $2.4\times10^{-6}$ for $N_{\min}=200$ and $6.4\times10^{-6}$ for
$N_{\min}=120$, and by at most $8.5\times10^{-6}$ with $e_r=0$ and
$N_{\min}=200$. The model-free slopes of $\log\E[\tau_N^{\mathrm{off}}]$ between
$N=380$ and $N=400$ lie $0.05\%$ to $0.11\%$ below $\Delta V_{\mathrm{num}}(r)$
and still increase with $N$. Halving both sides of the rectangle changes
$\E[\tau_N^{\mathrm{off}}]$ by at most $2.2\%$ at $N=40$ and by less than
$0.05\%$ for $N\ge100$. This is expected, since $R_{\mathrm{off}}$ lies deep in
the off basin, where arrival costs no additional action.

The simulated means agree with these values. Over the $34$ pairs $(r,N)$ of the
second campaign, the standardized differences between
$\log\widehat{\E[\tau_N^{\mathrm{off}}]}$ and its exact value have mean $0.10$
and standard deviation $0.97$, the largest is $2.7$ in absolute value, and
their sum of squares is $31.3$ ($p=0.60$ for $\chi^2$ with $34$ degrees of
freedom). Hence the solver and the simulator, which share no code, validate
each other. Applied to the simulated means, the regression
\eqref{eq:regression} gives the slopes
$\Delta V^{\mathrm{off}}_{\mathrm{SSA}}(r)$ of Table~\ref{tab:exit-barriers},
with reduced chi-square between $0.40$ and $1.66$. They lie $1.3$, $0.2$, $0.6$,
$-0.5$ and $0.6$ fitted standard errors from $\Delta V_{\mathrm{num}}(r)$ at
$r=0.5,1,2,4,8$, and they stay within $1.4$ standard errors of it at every rate
when the smallest size, or the two smallest, are removed.
Without the smallest size, the simulations still exclude $\Delta V_\infty$ at
$r=0.5$ and $1$ ($p\le0.001$), but not at $r=2$, $4$ and $8$ ($p\ge0.20$);
there the exclusion of $\Delta V_\infty$ rests on the backward equation.

The regression model matters. With $c_r=0$ every slope lies below
$\Delta V_{\mathrm{num}}(r)$, by $0.0012$ to $0.0033$, i.e., by $2.3$ to $7.4$
fitted standard errors. This bias is not statistical. Fitted to the exact means
at the simulated sizes, with the simulation weights, the model with $c_r=0$
returns $\Delta V_{\mathrm{num}}(r)$ minus $0.0017$ to $0.0027$, and \eqref{eq:regression}
returns it within $4\times10^{-4}$. The
second campaign required about $1.3\times10^5$ core-seconds.

\begin{figure}[tbp]
 \centering
 \includegraphics[width=\textwidth]{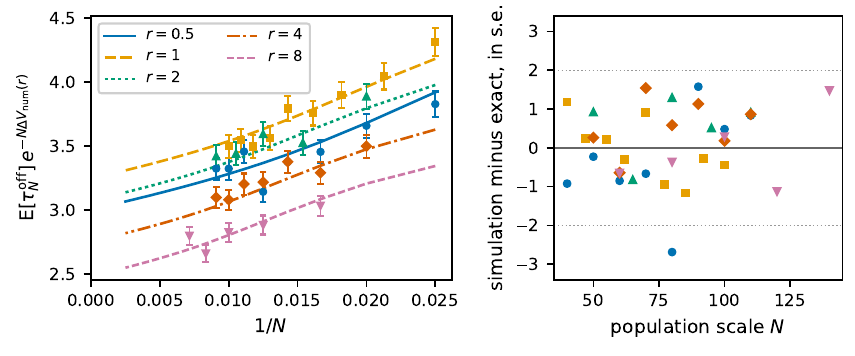}
 \caption{Mean extinction times \eqref{eq:collapse-time-off} at $c=0.36$. Left:
 the prefactor $\E[\tau_N^{\mathrm{off}}]\,e^{-N\Delta V_{\mathrm{num}}(r)}$
 against $1/N$; curves are the backward equation for $40\le N\le400$, and points
 are the simulated means with nominal $n_{\mathrm{exit}}^{-1/2}$ error bars.
 Right: the difference between the simulated and the exact
 $\log\E[\tau_N^{\mathrm{off}}]$, in units of the nominal standard error, with
 the same markers. All quantities are dimensionless.}
 \label{fig:extinction-scaling}
\end{figure}

\subsubsection*{A cheaper threshold point and recovery}

The stopping rule \eqref{eq:threshold-time} allows every signal value, whereas
\eqref{eq:target} fixes both coordinates. The corresponding variational
quantity is the interior threshold barrier
$\Delta V_{\mathrm{thr}}(r):=\inf_{w>0}V_r^Q((x^\ast,w))\le\Delta V(r)$. By
Proposition~\ref{prop:saddle-off} and the remark after it, the points
$(x^\ast,w)$ with $w<w^\ast$, which lie in the off basin, are no cheaper than
the saddle; the points with $w>w^\ast$ lie in the cooperative basin, and
whether one of them is cheaper has to be checked. We computed the profile
$w\mapsto V_r^Q((x^\ast,w))$ with the method of Section~\ref{sec:mam}, the
endpoint now fixed at $(x^\ast,w)$, at $r\in\{0.5,1,2,4,8\}$ for
$w^\ast<w\le w^\ast+5$, from several initial paths, with
$T\in\{40,80,160\}$ and $\Delta t\in\{0.04,0.02,0.01\}$, the paths being
constrained to $x\ge x^\ast$ before the endpoint (which does not change the
infimum over the line, since $\mathbb L\ge0$). At $r=1,2,4,8$ the computed
profile increases in $w$ and no computed point costs less than the saddle, so
there we take the computed threshold barrier to be $\Delta V_{\mathrm{num}}(r)$;
this is a computed agreement at these rates, not a statement for all $r\ge1$.
At $r=0.5$ the profile has an interior minimum at $w=1.808$, i.e., $0.69$ above
$w^\ast$, with action $0.11308$ against $0.12248$ at the saddle. The difference,
$9.4\times10^{-3}$, compares with a variation below $10^{-6}$ under the
refinements, and an independently written solver reproduces both values to
$10^{-6}$. The computed threshold value at $r=0.5$ is therefore $7.7\%$ below
the computed saddle value, and we write $\Delta V_{\mathrm{thr}}(0.5)\le0.1131$
for this computed upper estimate. At small $r$ the cheapest threshold crossing
keeps the signal high, and the deterministic flow then returns to
$U_{\mathrm{on}}$.

For the threshold time the stopping rule in counts is
$X^N\le\lfloor Nx^\ast\rfloor/N$, i.e., the effective threshold lies
$\{Nx^\ast\}/N$ below $x^\ast$, where
$\{Nx^\ast\}=Nx^\ast-\lfloor Nx^\ast\rfloor$ oscillates with $N$. Accordingly,
we fit the exact threshold times of the backward equation by
\eqref{eq:bke-fit} with $e_r=0$ and an additional term $g_r\{Nx^\ast\}$, over
$N\ge200$ (Table~\ref{tab:threshold-barriers}). These fits are rougher than
those of the extinction time: the residuals are $1.3\times10^{-3}$ to
$2.7\times10^{-3}$, and the slope still moves by a few $10^{-4}$ with
$N_{\min}$. Nevertheless, at $r=0.5$ the threshold time grows at the rate
$0.1124$, which is $0.6\%$ below the computed upper estimate $0.1131$ of
$\Delta V_{\mathrm{thr}}(0.5)$ and $8\%$ below the saddle value $0.1225$. At
$r\ge1$ the computed threshold barrier is attained at the saddle, where the
first variation of the action along the line vanishes, hence the lattice
coefficient there is a finite-$N$ effect; indeed it decreases with $N_{\min}$,
$g_1=0.163,0.152,0.143$ for $N_{\min}=40,120,200$, and it is consistent with
the first crossings at $r=1$, which have $w>w^\ast$ in $99\%$ of trajectories.
The simulated threshold means agree with the exact ones: over $34$ pairs each,
the standardized differences have mean $-0.19$ and standard deviation $0.75$ in
the first campaign, and $0.25$ and $0.91$ in the second. Simulations with the threshold
shifted by whole lattice steps are kept in the repository~\cite{qsmmsccode}.

Threshold crossings are often followed by recovery. Of all crossings, the
fraction that returns to the box $[x_{\mathrm{on}}\pm0.15]\times
[\bar w(x_{\mathrm{on}})\pm0.30]$ before reaching $R_{\mathrm{off}}$ is
$0.86$--$0.97$ at $r=0.5$ and $0.37$--$0.46$ at $r=8$, generally growing with
$N$. At $r=0.5$ the signal at the first crossing exceeds $w^\ast$ in more than
$99\%$ of trajectories, with median excess between $0.68$ and $0.83$, close to
the minimizing endpoint above; at $r=8$ it does so in $81$--$83\%$. Such
crossings lie in the cooperative basin (Section~\ref{sec:off-target}). At
$r=0.5$ the probability that a crossing with $w>w^\ast$ reaches
$R_{\mathrm{off}}$ before returning falls from $0.14$ at $N=40$ to $0.03$ at
$N=110$. Crossings with $w\le w^\ast$, which lie in the off basin, still return
in $15$--$40\%$ of cases at every rate, with no clear trend in $N$.

Finally, we test whether the success probability near the saddle stays of order
one, as the attempt argument of Section~\ref{sec:barrier} requires. We start the
process at the lattice point nearest to $U_\ast+sN^{-1/2}v_u$, where $v_u$ is
the unit unstable eigenvector at $U_\ast$ and $s\in\{-1,0,1\}$, and stop at the
first entry into $R_{\mathrm{off}}$ or into the box around $U_{\mathrm{on}}$,
with $20\,000$ trajectories for each $r\in\{0.5,1,2,4,8\}$ and
$N\in\{50,100,200,\dots,3200\}$. Started at the saddle, the probability of entering
$R_{\mathrm{off}}$ first decreases from $0.63$--$0.64$ at $N=50$ to
$0.51$--$0.52$ at $N=3200$, with $95\%$ intervals of half-width at most
$0.007$; at $N=3200$ it is $0.74$--$0.79$ for $s=-1$ and $0.23$--$0.29$ for
$s=1$. Over this range it is therefore of order one, although the finite range
of $N$ cannot establish the limit.

%=====================================================================
\section{Finite \texorpdfstring{$r$}{r} raises the extinction barrier}
\label{sec:results}

Figure~\ref{fig:action-results} shows the minimum-action results. The left
panel exposes the mechanism: finite-$r$ minimum-action paths bow above the
quasi-steady signal relation and approach it as $r$ increases. The middle
panel shows the corresponding barriers. At $r=1$ the barrier is $0.0953$,
i.e., it exceeds the eliminated-signal value $\Delta V_\infty=0.0614$ by $55\%$,
and between $r=0.5$ and $r=64$ it changes by a factor of $1.98$, although the deterministic equilibria and their stability types are
independent of $r$. The right panel tests the exponent in
\eqref{eq:crossover-law} without using the analytical coefficient: for each of
five costs, $r[\Delta V_{\mathrm{num}}(r)-\Delta V_\infty]$ approaches a plateau.
Over $r\ge4$ the plateau spread is $0.57\%$ at $c=0.36$ and at most $3.5\%$
across the five costs. The coefficients extracted through \eqref{eq:fit} agree
with the formal prediction \eqref{eq:A-theory} to relative errors
$0.56\%$, $0.36\%$, $0.24\%$, $0.08\%$ and $0.01\%$ as $c$ ranges from $0.20$ to
$0.50$.

The calculation of Section~\ref{ssec:finite-r} separates the lag of the signal
from its noise. At $c=0.36$, a signal which follows its drift without
fluctuations raises the barrier by $A_{\mathrm{lag}}/r$ with
$A_{\mathrm{lag}}=0.0472$, and the fluctuations of a signal as abundant as the
cells remove $25\%$ of this increase, leaving $A=0.0354$. We test this by
scaling the signal count as $\nu N$. At $r=64$ the plateau
$r[\Delta V_\nu(r)-\Delta V_\infty]$ equals $0.03553$, $0.04140$, $0.04433$,
$0.04653$ and $0.04709$ for $\nu=1,2,4,16,64$, against
$A(\nu)=0.03544$, $0.04132$, $0.04426$, $0.04646$ and $0.04701$ from
\eqref{eq:A-nu}, i.e., a difference of $0.16\%$ to $0.24\%$, which is the size of
the $O(1/r)$ remainder at $\nu=1$. Hence the barrier is raised by the lag, and
the signal noise offsets a part of it which vanishes as $\nu\to\infty$.

The expansion \eqref{eq:crossover-law} concerns $r\to\infty$. Compared with the
computed barriers, $\Delta V_\infty+A/r$ is $0.06\%$ too large at $r=2$,
$1.6\%$ at $r=1$ and $8.0\%$ at $r=0.5$, and it diverges as $r\to0$. For small
$r$ the roles of the two variables are exchanged. A formal elimination of the
cells, $x=x(w):=(b(w)-d_0-c)/\theta$, leaves a birth-death process for the
signal, whose barrier is
\[
 V_0=\int_{w^\ast}^{\bar w(x_{\mathrm{on}})}
 \log\frac{\alpha_\Coop\,x(w)}{\kappa w}\,dw=0.2724 .
\]
The open points in the middle panel of Figure~\ref{fig:action-results} are the
slow-signal sweep of Section~\ref{sec:mam}. The barrier increases monotonically
as $r$ decreases, from $0.1225$ at $r=0.5$ to $0.2670$ at $r=0.0078$, and
$V_0-\Delta V(r)$ is close to linear in $r$: the ratio $[V_0-\Delta V(r)]/r$
equals $0.63$, $0.67$, $0.68$ and $0.69$ at $r=0.0625$, $0.031$, $0.016$ and
$0.0078$. These computations support $V_0$ as the limit of $\Delta V(r)$ as
$r\to0$, with a correction of order $r$; we do not prove this. Over all the tested range of $r$ the
barrier therefore varies between $0.061$ and $0.272$, i.e., by a factor $4.4$,
with the equilibria fixed.

The barrier of the diffusion approximation, i.e., of the quadratic Lagrangian
of Remark~\ref{rem:H-contains}, lies $0.07\%$ below $\Delta V_{\mathrm{num}}(r)$
at $r=0.5$ and $0.10\%$ to $0.11\%$ below it at the seven rates from $r=1$ to
$r=64$. Hence at these parameters the barrier, including its dependence on $r$,
is a small-momentum quantity: the jump structure enters through the gross rates
in $D$, and the higher cumulants contribute $0.1\%$. This concerns the exponent
only, and it does not show that the diffusion approximation gives the right
prefactor.

\begin{figure}[tbp]
 \centering
 \includegraphics[width=\textwidth]{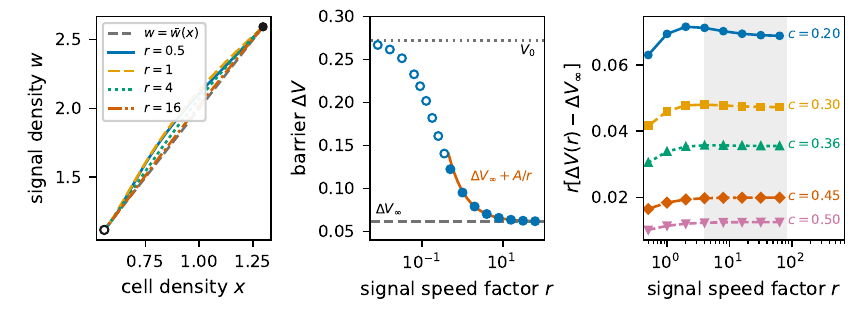}
 \caption{Minimum-action crossover. Left: minimum-action paths from
 $U_{\mathrm{on}}$ (filled point) to the saddle $U_\ast$ (open point) at
 $c=0.36$, with the endpoint imposed as in \eqref{eq:target}; the dashed line is
 the quasi-steady relation. Middle: finite-$r$
 barriers, with open points for the slow-signal sweep $r<0.5$, the
 eliminated-signal value $\Delta V_\infty$, the level $V_0$ obtained by
 eliminating the cells, and $\Delta V_\infty+A/r$ using
 \eqref{eq:A-theory}. Right: the rescaled excess barrier at five costs. A
 plateau identifies the $1/r$ scaling and its level gives $A_{\mathrm{num}}$;
 the shaded region marks $r\ge4$. All quantities are dimensionless.}
 \label{fig:action-results}
\end{figure}

Table~\ref{tab:exit-barriers} compares the growth rates of the mean extinction
time with the minimum-action saddle barrier. They test the prediction
\eqref{eq:lifetime-prediction} directly, including the time spent in failed
attempts. The rates $\Delta V_{\mathrm{BE}}(r)$ of the backward equation agree
with $\Delta V_{\mathrm{num}}(r)$ to within $2.4\times10^{-6}$ at all five rates,
although the two computations share nothing: the first solves linear systems
for the generator on the lattice, and the second minimizes the path action.
The distance from $\Delta V_{\mathrm{num}}(r)$ to $\Delta V_\infty$ is $0.0611$
at $r=0.5$ and $0.0045$ at $r=8$, i.e., at least three orders of magnitude
larger. Hence the eliminated-signal value is excluded at every rate, including
$r=4$ and $r=8$, where the two candidates differ by only $13\%$ and $7\%$. The
simulated slopes are consistent with $\Delta V_{\mathrm{num}}(r)$ at every rate.
By themselves they exclude $\Delta V_\infty$ with $p\approx10^{-6}$ or smaller at $r=0.5,1$, with
$p=0.016$ and $p=0.022$ at $r=2,4$, and not at $r=8$, where $p=0.10$. Notice
that all of this concerns $40\le N\le400$: the backward equation is a finite-$N$
computation, and the limit in \eqref{eq:lifetime-prediction} remains an
extrapolation. The minimization alone bounds each infimum from above, so it
cannot establish that the barrier exceeds $\Delta V_\infty$; what excludes the
eliminated-signal value is the backward equation, together with the simulation,
without using the action solver. The minimum-action calculation supplies in
return the mechanism and the $1/r$ law, neither of which the exclusion alone
would reveal.

\begin{table}[htbp]
\caption{Extinction barriers at $c=0.36$. $\Delta V_{\mathrm{num}}$:
minimum-action saddle barrier; $\Delta V_{\mathrm{BE}}$ and $C_r(\infty)$: growth
rate and limiting prefactor $e^{a_r}$ of the mean extinction time of the
backward equation, from \eqref{eq:bke-fit} with $N_{\min}=200$;
$\Delta V^{\mathrm{off}}_{\mathrm{SSA}}$: slope of the simulated extinction times
under \eqref{eq:regression}, with its reduced chi-square and degrees of
freedom; $c_r=0$: the same fit without the term $c_r/N$. The final column
gives the two-sided Student $p$-value of
$\Delta V^{\mathrm{off}}_{\mathrm{SSA}}-\Delta V_\infty$, with
$\Delta V_\infty=0.061379$.}
\label{tab:exit-barriers}
\centering
\small
\renewcommand{\arraystretch}{1.15}
\setlength{\tabcolsep}{3pt}
\begin{tabular}{@{}cccccccc@{}}
\toprule
$\boldsymbol{r}$ & \textbf{sizes} & $\boldsymbol{\Delta V_{\mathrm{num}}(r)}$ & $\boldsymbol{\Delta V_{\mathrm{BE}}(r)}$ & $\boldsymbol{C_r(\infty)}$ & $\boldsymbol{\Delta V^{\mathrm{off}}_{\mathrm{SSA}}(r)}$, $\boldsymbol{\chi^2_{\mathrm{red}}}$ \textbf{(dof)} & $\boldsymbol{c_r=0}$
& $\boldsymbol{p}$\\
\midrule
$0.5$ & $8$ & $0.122477$ & $0.122477$ & $3.00$ & $0.1249\pm0.0018$, $1.66$ $(5)$ & $0.1206\pm0.0007$ & $4\times10^{-7}$\\
$1$   & $9$ & $0.095259$ & $0.095258$ & $3.25$ & $0.0956\pm0.0017$, $0.40$ $(6)$ & $0.0919\pm0.0004$ & $1\times10^{-6}$\\
$2$   & $5$ & $0.079057$ & $0.079056$ & $3.08$ & $0.0806\pm0.0025$, $1.06$ $(2)$ & $0.0772\pm0.0006$ & $0.016$\\
$4$   & $7$ & $0.070313$ & $0.070311$ & $2.76$ & $0.0693\pm0.0022$, $0.61$ $(4)$ & $0.0683\pm0.0005$ & $0.022$\\
$8$   & $5$ & $0.065843$ & $0.065840$ & $2.49$ & $0.0670\pm0.0019$, $1.35$ $(2)$ & $0.0646\pm0.0005$ & $0.10$\\
\bottomrule
\end{tabular}
\end{table}

At $r=0.5$ the extinction time does not follow the cheaper threshold point: its
growth rate $0.12248$ exceeds the computed threshold barrier
$\Delta V_{\mathrm{thr}}(0.5)\le0.1131$ by $0.0094$, whereas the threshold time
grows at the rate $0.1124$ (Table~\ref{tab:threshold-barriers}). At
$r=1,2,4,8$, where the computed threshold barrier equals the saddle value, the
growth rates of the threshold time agree with it within $4\times10^{-4}$. These
are finite-$N$ findings, and they support the conditional prediction without
proving it.

\begin{table}[htbp]
\caption{Threshold barriers at $c=0.36$. $\Delta V_{\mathrm{thr}}$: computed
minimum-action threshold barrier of Section~\ref{sec:ssa}, equal to the saddle
value for $r\ge1$ and an upper estimate at $r=0.5$;
$\Delta V^{\mathrm{thr}}_{\mathrm{BE}}$ and $g_r$: growth rate of the mean
threshold time of the backward equation and coefficient of the lattice term
$g_r\{Nx^\ast\}$, fitted over $N\ge200$; the last column is the root mean square
residual of that fit.}
\label{tab:threshold-barriers}
\centering
\small
\renewcommand{\arraystretch}{1.15}
\setlength{\tabcolsep}{6pt}
\begin{tabular}{@{}ccccc@{}}
\toprule
$\boldsymbol{r}$ & $\boldsymbol{\Delta V_{\mathrm{thr}}(r)}$ & $\boldsymbol{\Delta V^{\mathrm{thr}}_{\mathrm{BE}}(r)}$ & $\boldsymbol{g_r}$ & \textbf{residual}\\
\midrule
$0.5$ & $0.1131$ & $0.1124$ & $0.28$ & $0.0014$\\
$1$   & $0.0953$ & $0.0957$ & $0.14$ & $0.0027$\\
$2$   & $0.0791$ & $0.0792$ & $0.08$ & $0.0022$\\
$4$   & $0.0703$ & $0.0703$ & $0.06$ & $0.0016$\\
$8$   & $0.0658$ & $0.0659$ & $0.05$ & $0.0013$\\
\bottomrule
\end{tabular}
\end{table}

The $55\%$ excess and the coefficient $A$ are parameter-specific. For this
network, the $O(1/r)$ correction is observed at all five costs, and it is amplified
exponentially in the population scale.

%=====================================================================
\section{Conclusion}
\label{sec:conclusion}

Within its stated numerical accuracy, the calculation establishes the central
point: a parameter that leaves every deterministic equilibrium and stability
type unchanged can alter the extinction barrier, and hence, subject to the exit
estimates of Section~\ref{sec:barrier}, the mean extinction time by a factor
exponential in the population size. Eliminating the fast signal before constructing the action
removes the joint cell-signal history that produces this difference. Rare
events are therefore not determined by the deterministic landscape alone, and a
reduction that is harmless for the equilibria of a network need not be harmless
for its exit times.

The two numerical routes are consistent, but they do not remove every source of
uncertainty. The minimization of \eqref{eq:discrete-action} is nonconvex, hence
the value returned by the optimizer bounds the infimum in
\eqref{eq:target} from above, as the exact action of the interpolated paths
confirms at the five rates checked, but it need not attain it. The calibrations of Section~\ref{sec:mam} are local, in the
one-dimensional case and, through \eqref{eq:hessian}, at $U_{\mathrm{on}}$ at
finite $r$, so neither verifies that the finite-$r$ optimizer has found the
global minimum. Proposition~\ref{prop:saddle-off} identifies the saddle as
the least-cost entrance to the basin of extinction, and the agreement with the
threshold experiment at $r=1,2,4,8$ gives no indication of a cheaper route
there; neither shows that the computed path is a global minimizer. At $r=0.5$
the threshold computation shows that a different stopping set can have a
smaller computed barrier. Three limitations therefore remain separate: global
minimization of the action, the relation between the saddle barrier and
stopping sets other than $R_{\mathrm{off}}$, and the metastable exit estimates for the jump process,
including the effect of the coordinate axes. For the extinction time itself the
last is supported numerically by the backward equation and the simulations of
Section~\ref{sec:ssa}, for $40\le N\le400$ only. Likewise the finite-$T$ actions approach \eqref{eq:target} from above,
and Section~\ref{sec:mam} shows no change at the reported precision
between $T=40$ and $T=60$, which controls the horizon truncation numerically
but is not an analytical error bound.

Over $40\le N\le400$ the prefactor $C_r(N)$ in \eqref{eq:exit-asymptotic} tends
to a constant between $2.5$ and $3.3$, with a correction of relative order
$1/N$. A formula for this constant would require an Eyring-Kramers analysis for
this non-reversible jump process, which, to our knowledge, is not available in the literature considered here. Likewise, the expansion
\eqref{eq:crossover-law} with the coefficient \eqref{eq:A-theory} and the limit
$V_0$ at small $r$ are supported numerically and not rigorously proved. Finally, the
generic momentum inversion and envelope gradient extend directly to more
channels and coordinates, but in higher dimensions additional saddles introduce
distinct local minima, and such networks require multiple initial paths and,
where necessary, a geometric minimum-action or string
method~\cite{heymann2008geometric}.

\appendix
\renewcommand{\theequation}{\thesection.\arabic{equation}}
\setcounter{equation}{0}
\section{Proof of Proposition~\ref{prop:saddle-off}}
\label{app:action}

Throughout, $c\in[0.2,0.5]$ and $r>0$ are fixed, $\Phi_t$ denotes the flow of
\eqref{eq:mean-field}, and $V=V_r^Q$. Recall from Section~\ref{sec:off-target}
that $b_0<d_0+c$, that $\bar F$ has exactly two positive zeros, both simple,
and that the equilibria in $[0,\infty)^2$ are $U_{\mathrm{off}}$, the
hyperbolic saddle $U_\ast$ and $U_{\mathrm{on}}$. At $U_\ast$ the Jacobian is
\[
 J_r=\begin{pmatrix}
 -\theta x^\ast & x^\ast b'(w^\ast)\\
 r\alpha_\Coop & -r\kappa
 \end{pmatrix},
\]
with eigenvalues $\mu_s<0<\mu_u$. An eigenvector for $\mu$ is
$(x^\ast b'(w^\ast),\,\mu+\theta x^\ast)$, which has two positive components
for $\mu_u$.

\begin{lemma}
\label{lem:basin-facts}
\begin{enumerate}
 \item[(i)] Every solution started in $Q$ converges to $U_{\mathrm{off}}$,
 $U_\ast$ or $U_{\mathrm{on}}$. Consequently
 $\partial\mathcal B_{\mathrm{on}}(r)\cap Q\subset W^s(U_\ast)$.
 \item[(ii)] The branch of the unstable manifold of $U_\ast$ that leaves
 towards $x<x^\ast$ lies in $Q$ and converges to $U_{\mathrm{off}}$.
\end{enumerate}
\end{lemma}

\begin{proof}
(i) Let $u\in Q$. Since $\dot x\le x(b_0+b_1-\theta x)$ and
$\dot w\le r(\alpha_\Coop x-\kappa w)$, the solution is bounded, so its
$\omega$-limit set $\omega(u)$ is nonempty, compact, connected and invariant.
If $\omega(u)$ contains a point $(0,w)$, it contains the solution through it,
which converges to $U_{\mathrm{off}}$ along the axis; if it contains $(x,0)$
with $x>0$, the backward solution through that point leaves $[0,\infty)^2$,
since $\dot w=r\alpha_\Coop x>0$ there, which is impossible. Since the stable
equilibria attract a neighbourhood, $\omega(u)$ is then $\{U_{\mathrm{off}}\}$,
$\{U_{\mathrm{on}}\}$, or a subset of $Q$ that contains no equilibrium other
than $U_\ast$.
In the last case, with the Dulac function $g=1/x$,
\[
 \operatorname{div}(gF)
 =\partial_x\bigl[b(w)-d(x)\bigr]
 +\partial_w\Bigl[\frac{r(\alpha_\Coop x-\kappa w)}{x}\Bigr]
 =-\theta-\frac{r\kappa}{x}<0
\]
on $Q$, which is convex, hence $Q$ contains neither a periodic orbit nor a
homoclinic loop at $U_\ast$, and by the Poincar\'e-Bendixson theorem
$\omega(u)=\{U_\ast\}$. Finally, the basins of the two stable equilibria are
open and disjoint, so a point of $\partial\mathcal B_{\mathrm{on}}(r)\cap Q$
lies in neither of them and converges to $U_\ast$.

(ii) For $0\le x\le x^\ast$ we have
$b_0<d(0)\le d(x)\le d(x^\ast)=b(w^\ast)<b_0+b_1$, hence
$\zeta:=b^{-1}\circ d$ is well defined, positive and $C^1$ on $[0,x^\ast]$,
with $\zeta(x^\ast)=w^\ast$ and $\zeta'=\theta/b'(\zeta)$; the curve
$w=\zeta(x)$ is the nullcline $\dot x=0$. Since $\bar F<0$ on $(0,x^\ast)$ and
$b_0<d_0+c$, also $\bar w(x)<\zeta(x)$ on $[0,x^\ast)$. Let
$\mathcal O:=\{(x,w):0<x<x^\ast,\ 0<w<\zeta(x)\}$. The edge $x=0$ is
invariant; on the remaining open boundary arcs the field points inward: on $w=\zeta(x)$ we have $\dot x=0$ and
$\dot w=r\kappa[\bar w(x)-\zeta(x)]<0$, on $x=x^\ast$ with $w<w^\ast$ we have
$\dot x=x^\ast[b(w)-b(w^\ast)]<0$, and on $w=0$ we have $\dot w>0$. Hence
$\mathcal O$ is forward invariant. Inside $\mathcal O$, $\dot x<0$. By (i),
every trajectory in $\mathcal O$ converges to an equilibrium; since
$x(t)\le x(0)<x^\ast$, that equilibrium must be $U_{\mathrm{off}}$. Near $U_\ast$ the unstable
branch towards $x<x^\ast$ has slope
$s_u=(\mu_u+\theta x^\ast)/(x^\ast b'(w^\ast))$, and
$s_u-\zeta'(x^\ast)=\mu_u/(x^\ast b'(w^\ast))>0$. Hence it lies below the
nullcline for $x<x^\ast$ close to $x^\ast$, i.e., in $\mathcal O$, and it
converges to $U_{\mathrm{off}}$.
\end{proof}

\begin{lemma}
\label{lem:action-properties}
\begin{enumerate}
 \item[(i)] $\mathbb L(U,v)\ge0$ and $\mathbb L(U,F(U))=0$ for $U\in Q$ and
 $v\in\R^2$. Consequently $V(u,\Phi_t(u))=0$ for $u\in Q$ and $t>0$.
 \item[(ii)] $V(u,z)\le V(u,y)+V(y,z)$ for $u,y,z\in Q$.
 \item[(iii)] For every compact convex $K\subset Q$ there is $C_K<\infty$
 such that $V(u,z)\le C_K\lvert u-z\rvert$ for $u,z\in K$.
\end{enumerate}
\end{lemma}

\begin{proof}
(i) Taking $p=0$ in the supremum gives $\mathbb L\ge0$, since
$\mathbb H(U,0)=0$. By convexity and $\nabla_p\mathbb H(U,0)=F(U)$
(Remark~\ref{rem:H-contains}), $\mathbb H(U,p)\ge\inner{F(U)}{p}$, hence
$\mathbb L(U,F(U))=0$, and the deterministic solution, which stays in $Q$, has
zero action.
(ii) Concatenate two paths whose actions are within $\varepsilon$ of $V(u,y)$
and $V(y,z)$, and let $\varepsilon\to0$.
(iii) On $K$ the four intensities are continuous and bounded away from zero,
hence by \eqref{eq:explicit-legendre} and \eqref{eq:explicit-L} the local
rate $\mathbb L$ is bounded by some $C_K$ on $K\times\{\lvert v\rvert\le1\}$.
The segment from $u$ to $z$ traversed at unit speed stays in $K$ and has action
at most $C_K\lvert u-z\rvert$.
\end{proof}

\begin{proof}[Proof of Proposition~\ref{prop:saddle-off}]
Let $K\subset Q$ be a closed ball of radius $\rho>0$ centred at $U_\ast$, and
let $C_K$ be as in Lemma~\ref{lem:action-properties}(iii).

\emph{Points of the separatrix are no cheaper than the saddle.} Let
$y\in W^s(U_\ast)\cap Q$ and $0<\varepsilon<\rho$, and choose $t>0$ with
$\lvert\Phi_t(y)-U_\ast\rvert\le\varepsilon$. By
Lemma~\ref{lem:action-properties},
\[
 \Delta V(r)
 \le V(y)+V\bigl(y,\Phi_t(y)\bigr)+V\bigl(\Phi_t(y),U_\ast\bigr)
 \le V(y)+C_K\varepsilon ,
\]
and $\varepsilon\to0$ gives $V(y)\ge\Delta V(r)$.

\emph{Lower bound.} Let $z\in R\cap Q$ and let $U\in\mathrm{AC}([0,T];Q)$ run
from $U_{\mathrm{on}}$ to $z$. Since $U_{\mathrm{on}}$ lies in the open set
$\mathcal B_{\mathrm{on}}(r)$ and $z$ does not,
$t_1:=\inf\{t:U(t)\notin\mathcal B_{\mathrm{on}}(r)\}$ satisfies $0<t_1\le T$,
and $y:=U(t_1)\in\partial\mathcal B_{\mathrm{on}}(r)\cap Q$, hence
$y\in W^s(U_\ast)$ by Lemma~\ref{lem:basin-facts}(i). Since $\mathbb L\ge0$,
\[
 \mathcal S_T(U)\ge\mathcal S_{t_1}\bigl(U|_{[0,t_1]}\bigr)\ge V(y)\ge\Delta V(r),
\]
and taking the infimum over $T$, $U$ and $z$ gives
$\inf_{z\in R\cap Q}V(z)\ge\Delta V(r)$.

\emph{Upper bound.} Let $\Gamma$ be the unstable branch of
Lemma~\ref{lem:basin-facts}(ii). Given $0<\varepsilon<\rho$, choose
$u_\varepsilon\in\Gamma$ with $\lvert u_\varepsilon-U_\ast\rvert\le\varepsilon$.
Since $R$ is a relative neighbourhood of $U_{\mathrm{off}}$, there is $s>0$ with
$z_\varepsilon:=\Phi_s(u_\varepsilon)\in R\cap Q$, and by
Lemma~\ref{lem:action-properties}
\[
 V(z_\varepsilon)
 \le V(U_\ast)+V(U_\ast,u_\varepsilon)+V(u_\varepsilon,z_\varepsilon)
 \le\Delta V(r)+C_K\varepsilon .
\]
Hence $\inf_{z\in R\cap Q}V(z)\le\Delta V(r)+C_K\varepsilon$ for every small
$\varepsilon>0$, and $\varepsilon\to0$ gives the upper bound.
\end{proof}

\bibliographystyle{plain}
\bibliography{references}

\end{document}